\documentclass[11pt]{article}%
\usepackage[fleqn]{amsmath}
\usepackage{amsmath,amsfonts,amssymb,graphicx,makeidx,amsthm}
\usepackage{amscd,amsfonts,amssymb,multicol}
\usepackage{amssymb, amsmath}
\usepackage{amsfonts}
\usepackage{amssymb}
\usepackage{graphicx}%
\providecommand{\U}[1]{\protect\rule{.1in}{.1in}}
\newtheorem{theorem}{Theorem}[section]

\newtheorem{example}[theorem]{Example}

\newtheorem{problem}[theorem]{Problem}
\newtheorem{proposition}[theorem]{Proposition}
\newtheorem{remark}[theorem]{Remark}

\numberwithin{equation}{section}

\begin{document}

\title{\textbf{A direct method for solving inverse Sturm-Liouville problems}}
\author{Vladislav V. Kravchenko$^{1}$, Sergii M. Torba$^{1}$\\{\small $^{1}$ Departamento de Matem\'{a}ticas, Cinvestav, Unidad
Quer\'{e}taro, }\\{\small Libramiento Norponiente \#2000, Fracc. Real de Juriquilla,
Quer\'{e}taro, Qro., 76230 MEXICO.}\\{\small e-mail: vkravchenko@math.cinvestav.edu.mx,
storba@math.cinvestav.edu.mx, \thanks{Research was supported by CONACYT,
Mexico via the projects 222478 and 284470. Research of Vladislav Kravchenko
was supported by the Regional mathematical center of the Southern Federal
University, Russia.}}}
\maketitle

\begin{abstract}
We consider two main inverse Sturm-Liouville problems:\ the problem of
recovery of the potential and the boundary conditions from two spectra or from
a spectral density function. A simple method for practical solution of such
problems is developed, based on the transmutation operator approach, new
Neumann series of Bessel functions representations for solutions and the
Gelfand-Levitan equation. The method allows one to reduce the inverse
Sturm-Liouville problem directly to a system of linear algebraic equations,
such that the potential is recovered from the first element of the solution
vector. We prove the stability of the method and show its numerical efficiency
with several numerical examples.

\end{abstract}

\section{Introduction}

We consider classical inverse Sturm-Liouville problems on a finite interval.
Several techniques have been proposed for their numerical solution (see
\cite{Brown et al 2003}, \cite{Drignei 2015}, \cite{Gao et al 2013}, \cite{Gao
et al 2014}, \cite{IgnatievYurko}, \cite{Kammanee Bockman 2009},
\cite{Kr2019JIIP}, \cite{Lowe et al 1992}, \cite{Neamaty et al 2017},
\cite{Neamaty et al 2019}, \cite{Rohrl}, \cite{Rundell Sacks}, \cite{Sacks}). However,
usually the methods proposed require either the knowledge of additional
parameters like, e.g., the mean of the potential or do not offer a complete
solution of the problem, e.g., they may not be able to recover the boundary
conditions. Often both mentioned drawbacks are present. For example, one of
the most successful algorithms, proposed in \cite{Rundell Sacks}, requires an
estimate of the mean of the potential as well as the knowledge of the boundary
conditions. Moreover, the authors write \cite[p. 177]{Rundell Sacks}:
\textquotedblleft As was mentioned earlier, if complete spectral data is
available, then it is in theory possible to determine the boundary conditions
as part of the solution of the problem. We do not believe, however, that this
is numerically feasible in most cases.\textquotedblright\ The boundary
conditions are assumed to be known in \cite{Brown et al 2003}, \cite{Gao et al
2014}, \cite{Kammanee Bockman 2009} and \cite{Rohrl}. The method from
\cite{IgnatievYurko} as many other approaches requires the knowledge of a
parameter $\omega$ related to the mean of the potential and participating in
the second term of the asymptotics of the eigenvalues.

The approach developed in the present work allows one to obtain a complete
solution of the inverse Sturm-Liouville problem without requiring any
additional data. Moreover, we show that the parameter $\omega$ can be obtained
on a first step without using the asymptotics of the eigenvalues. The method
is based on the idea presented in \cite{Kr2019JIIP} and \cite{KrBook2020}
where it was shown that the inverse Sturm-Liouville problem of recovering the
problem by its spectral density function (Problem 1) can be reduced directly
to a system of linear algebraic equations by making use of the Gelfand-Levitan
integral equation and the Fourier-Legendre series expansion of the
transmutation integral kernel obtained in \cite{KNT}. Moreover, for recovering
the potential and the boundary conditions there is no need to solve the
Gelfand-Levitan equation and find the transmutation kernel. The first
coefficient of the Fourier-Legendre series which corresponds to the first
component of the solution vector of the linear algebraic system is sufficient
for recovering the Sturm-Liouville problem. In the present work we improve the
approach from \cite{Kr2019JIIP}, \cite{KrBook2020} by using the representation
for the Gelfand-Levitan kernel $F(x,t)$ obtained recently in \cite{KKT2019AMC}%
. We prove the stability of the method and extend it onto the inverse
Sturm-Liouville problem by two spectra (Problem 2).

The series representation for $F(x,t)$ from \cite{KKT2019AMC} to the
difference from the classical representation \cite{GL} does not present a jump
discontinuity and, moreover, converges uniformly on the whole domain of
definition. However, it requires the knowledge of the parameter $\omega$. We
show that as a consequence of the Fourier-Legendre series representation of
the transmutation kernel the parameter $\omega$ can be efficiently computed
from first eigenvalues without exploiting their asymptotics. We obtain a
system of linear algebraic equations for the coeffcients of the
Fourier-Legendre series based on the representation for $F(x,t)$ from
\cite{KKT2019AMC}, show that its truncated version is uniquely solvable, and
the numerical process is stable. Besides, for a still more accurate solution
of Problem 1 some auxiliary techniques are presented. The first allows us to
obtain the \textquotedblleft flipped\textquotedblright problem, that is the
Problem 1 for the flipped potential $q(\pi-x)$ with the boundary conditions
being interchanged. We show how the norming constants of this problem are
obtained from the spectral data of the original problem, and this result is
also based on the Fourier-Legendre series representation of the transmutation
kernel. The second technique serves for computing the second term of
asymptotics of the norming constants as well as some further asymptotic terms
of the eigenvalues.

Another result of the present work is the reduction of Problem 2 to Problem 1
which gives an excellent numerical method for solving the two spectra problem.
 The reduction uses again the Fourier-Legendre series expansion of the
transmutation integral kernel.

The method developed here is accurate and fast and allows one to obtain a
complete solution of the inverse Sturm-Liouville problem which includes not
only the potential but also the boundary conditions. We illustrate its
numerical performance by several examples including a smooth, non-smooth and a
discontinuous potentials.

\section{Preliminaries}

\subsection{Spectral data}

Let $q\in L_{2}(0,\pi)$ be real valued. Consider the Sturm-Liouville equation
\begin{equation}
-y^{\prime\prime}+q(x)y=\rho^{2}y,\quad x\in(0,\pi), \label{SL equation}%
\end{equation}
where $\rho\in\mathbb{C}$, and two sets of boundary conditions%
\begin{equation}
y^{\prime}(0)-hy(0)=0,\quad y^{\prime}(\pi)+Hy(\pi)=0, \label{bc1}%
\end{equation}
and%
\begin{equation}
y^{\prime}(0)-hy(0)=0,\quad y(\pi)=0, \label{bc2}%
\end{equation}
where $h$ and $H$ are arbitrary real constants.

Denote
\begin{equation}
\omega:=h+H+\frac{1}{2}\int_{0}^{\pi}q(t)\,dt \label{omega}%
\end{equation}
and
\begin{equation}
\omega_{1}:=h+\frac{1}{2}\int_{0}^{\pi}q(t)\,dt=\omega-H. \label{omega1}%
\end{equation}

By $\varphi(\rho,x)$ we denote a solution of (\ref{SL equation}) satisfying
the initial conditions%
\begin{equation}
\varphi(\rho,0)=1\quad\text{and}\quad\varphi^{\prime}(\rho,0)=h.
\label{init cond}%
\end{equation}
Obviously, for all $\rho\in\mathbb{C}$ the function $\varphi(\rho,x)$ fulfills
the first boundary condition, $\varphi^{\prime}(\rho,0)-h\varphi(\rho,0)=0$,
and thus, the spectrum of problem (\ref{SL equation}), (\ref{bc1}) is the
sequence of numbers $\left\{  \lambda_{n}=\rho_{n}^{2}\right\}  _{n=0}%
^{\infty}$ such that
\[
\varphi^{\prime}(\rho_{n},\pi)+H\varphi(\rho_{n},\pi)=0
\]
while the spectrum of problem (\ref{SL equation}), (\ref{bc2}) is the sequence
of numbers $\left\{  \nu_{n}=\mu_{n}^{2}\right\}  _{n=0}^{\infty}$ such that
\[
\varphi(\mu_{n},\pi)=0.
\]

Denote
\begin{equation}
\label{alpha n}\alpha_{n}:=\int_{0}^{\pi}\varphi^{2}(\rho_{n},x)\,dx.
\end{equation}
The set $\left\{  \alpha_{n}\right\}  _{n=0}^{\infty}$ is referred to as the
sequence of norming constants of problem (\ref{SL equation}), (\ref{bc1}).

Let us formulate two basic inverse Sturm-Liouville problems.

\begin{problem}
[Recovery of a Sturm-Liouville problem from its spectral function]%
\label{Problem 1} Given two sequences of real numbers $\left\{  \lambda
_{n}\right\}  _{n=0}^{\infty}$ and $\left\{  \alpha_{n}\right\}
_{n=0}^{\infty}$, find a real valued $q\in L_{2}(0,\pi)$, and the constants
$h$, $H\in\mathbb{R}$, such that $\left\{  \lambda_{n}\right\}  _{n=0}%
^{\infty}$ be the spectrum of problem \eqref{SL equation}, \eqref{bc1} and
$\left\{  \alpha_{n}\right\}  _{n=0}^{\infty}$ its sequence of norming constants.
\end{problem}

\begin{problem}
[Recovery of a Sturm-Liouville problem from two spectra]\label{Problem 2}
Given two sequences of real numbers $\left\{  \lambda_{n}\right\}
_{n=0}^{\infty}$ and $\left\{  \nu_{n}\right\}  _{n=0}^{\infty}$, find a real
valued $q\in L_{2}(0,\pi)$, and the constants $h$, $H\in\mathbb{R}$, such that
$\left\{  \lambda_{n}\right\}  _{n=0}^{\infty}$ be the spectrum of problem
\eqref{SL equation}, \eqref{bc1} and $\left\{  \nu_{n}\right\}  _{n=0}%
^{\infty}$ the spectrum of problem \eqref{SL equation}, \eqref{bc2}.
\end{problem}

In order to assure the existence of a solution, the sequences of numbers
appearing in the formulation of these two problems can not be completely
arbitrary. The following criteria are valid.

\begin{theorem}
[{see \cite[Theorem 1.3.2]{Yurko2007}}]The sequences of real numbers $\left\{
\lambda_{n}\right\}  _{n=0}^{\infty}$ and $\left\{  \alpha_{n}\right\}
_{n=0}^{\infty}$ represent spectral data of a Sturm-Liouville problem
\eqref{SL equation}, \eqref{bc1} if and only if the following relations hold%
\begin{equation}
\rho_{n}:=\sqrt{\lambda_{n}}=n+\frac{\omega}{\pi n}+\frac{k_{n}}{n}%
,\quad\alpha_{n}=\frac{\pi}{2}+\frac{K_{n}}{n},\quad\left\{  k_{n}\right\}
,\ \left\{  K_{n}\right\}  \in\ell_{2}, \label{asympt 1}%
\end{equation}
and
\[
\alpha_{n}>0,\quad\lambda_{n}\neq\lambda_{m}\ (n\neq m)
\]

\end{theorem}

\begin{theorem}
[{see \cite[Theorem 1.3.4]{Yurko2007} and \cite{SavchukShkalikov} for a
generalization}]The sequences of real numbers $\left\{  \lambda_{n}\right\}
_{n=0}^{\infty}$ and $\left\{  \nu_{n}\right\}  _{n=0}^{\infty}$ are spectra
of the Sturm-Liouville problems \eqref{SL equation}, \eqref{bc1} and
\eqref{SL equation}, \eqref{bc2}, respectively, if and only if the following
relations hold%
\begin{equation}
\rho_{n}:=\sqrt{\lambda_{n}}=n+\frac{\omega}{\pi n}+\frac{k_{n}}{n},\quad
\mu_{n}:=\sqrt{\nu_{n}}=n+\frac{1}{2}+\frac{\omega_{1}}{\pi n}+\frac{K_{n}}%
{n},\quad\left\{  k_{n}\right\}  ,\ \left\{  K_{n}\right\}  \in\ell
_{2},\label{asympt 2}%
\end{equation}
and%
\[
\lambda_{n}<\nu_{n}<\lambda_{n+1},\quad n\geq0.
\]

\end{theorem}

\subsection{The transmutation integral kernel and the Gelfand-Levitan
equation}

The solution $\varphi(\rho,x)$ admits the integral representation (see, e.g.,
\cite[Chapter 6]{Levitan Sargsian}, \cite[Chapter 1]{Marchenko}, \cite[Chapter
3]{SitnikShishkina Elsevier})%
\begin{equation}
\varphi(\rho,x)=\cos\rho x+\int_{0}^{x}G(x,t)\cos\rho
t\,dt\label{sol transmutation}%
\end{equation}
where the transmutation integral kernel $G$ is a continuous function of both
arguments in the domain $0\leq t\leq x\leq\pi$ and satisfies the equalities
\begin{equation}
G(x,x)=h+\frac{1}{2}\int_{0}^{x}q(t)\,dt\label{Gxx}%
\end{equation}
and
\[
\left.  \frac{\partial}{\partial t}G(x,t)\right\vert _{t=0}=0.
\]
It is of crucial importance that $G(x,t)$ is independent of $\rho$.

Let%
\begin{equation}
F(x,t)=\sum_{n=0}^{\infty}\left(  \frac{\cos\rho_{n}x\cos\rho_{n}t}{\alpha
_{n}}-\frac{\cos nx\cos nt}{\alpha_{n}^{0}}\right)  ,\quad0\leq t,\,x<\pi
\label{F}%
\end{equation}
where
\[
\alpha_{n}^{0}=%
\begin{cases}
\pi/2, & n>0,\\
\pi, & n=0.
\end{cases}
\]
The functions $G$ and $F$ are related by the Gelfand-Levitan equation
\begin{equation}
G(x,t)+F(x,t)+\int_{0}^{x}F(t,s)G(x,s)\,ds=0,\quad0<t<x<\pi.
\label{Gelfand-Levitan}%
\end{equation}
Due to a slow convergence of the series (\ref{F}) and the presence of a jump
discontinuity of the series (\ref{F}) at $x=t=\pi$ (we refer to
\cite{KKT2019AMC} for more details), it was proposed in \cite[Section
13.4]{KrBook2020} to work with the integrated version of the Gelfand-Levitan
equation
\begin{equation}
\int_{0}^{t}G(x,\tau)\,d\tau+\widetilde{F}(x,t)+\int_{0}^{x}G(x,s)\widetilde
{F}(s,t)\,ds=0,\quad0<t<x \label{Gelfand-Levitan integrated}%
\end{equation}
obtained from (\ref{Gelfand-Levitan}) by using the symmetry $F(t,s)=F(s,t)$ of
the function $F$ and integrating it with respect to $t$. Here
\begin{equation}
\widetilde{F}(x,t)=\int_{0}^{t}F(x,\tau)\,d\tau=\frac{\cos\rho_{0}x\sin
\rho_{0}t}{\alpha_{0}\rho_{0}}-\frac{t}{\alpha_{0}^{0}}+\sum_{n=1}^{\infty
}\left(  \frac{\cos\rho_{n}x\sin\rho_{n}t}{\alpha_{n}\rho_{n}}-\frac{\cos
nx\sin nt}{\alpha_{n}^{0}n}\right)  . \label{Ftilde orig}%
\end{equation}
The advantage of (\ref{Ftilde orig}) lies in a faster convergence of the series.

We will be especially interested in the case when $\rho_{0}=0$. Then
\begin{equation}
\widetilde{F}(x,t)=\left(  \frac{1}{\alpha_{0}}-\frac{1}{\alpha_{0}^{0}%
}\right)  t+\sum_{n=1}^{\infty}\left(  \frac{\cos\rho_{n}x\sin\rho_{n}%
t}{\alpha_{n}\rho_{n}}-\frac{\cos nx\sin nt}{\alpha_{n}^{0}n}\right)  .
\label{Ftilde}%
\end{equation}

However, the integrated Gelfand-Levitan equation is no more a Fredholm
equation of the second kind (with respect to the function $G(x,\cdot)$). As a
result, it is unclear whether it is possible to obtain stability results for
corresponding numerical schemes of approximate solution of
\eqref{Gelfand-Levitan integrated}. Recall that for Fredholm equations there
is a well-developed general theory, see, e.g., \cite[Chapter 14,
\S 4]{KantorovichAkilov}.

In \cite{KKT2019AMC} another representation for the function $F$ was derived,
based on the following idea. The convergence of the series \eqref{F} can be
improved by subtracting certain expressions termwise and eventually adding a
closed form expression equal to the infinite sum of subtracted functions.
Namely (see \cite[(25)]{KKT2019AMC}),
\begin{equation}%
\begin{split}
F(x,t)  &  =\sum_{n=1}^{\infty}\left(  \frac{\cos\rho_{n}x\,\cos\rho_{n}%
t}{\alpha_{n}}-\frac{\cos nx\,\cos nt}{\alpha_{n}^{0}}+\frac{2\omega}{\pi
^{2}n}\Bigl(x\sin nx\,\cos nt+t\sin nt\,\cos nx\Bigr)\right) \\
&  \quad+\frac{1}{\alpha_{0}}-\frac{1}{\pi}-\frac{\omega}{\pi^{2}}%
\bigl(\pi\max\{x,t\}-x^{2}-t^{2}\bigr),
\end{split}
\label{F alt}%
\end{equation}
here $\omega$ is the parameter given by \eqref{omega}.

The series in \eqref{F alt} has no jump discontinuity at $x=t=\pi$ and
converges uniformly and absolutely and much faster than \eqref{F}, allowing
one to work directly with the Gelfand-Levitan equation \eqref{Gelfand-Levitan}
without passing to its integrated version (\ref{Gelfand-Levitan integrated}).

\subsection{Fourier-Legendre series expansion of the transmutation kernel}

The following result from \cite{KNT} will be used.

\begin{theorem}
[\cite{KNT}]\label{Theorem G series representation} The integral transmutation
kernel $G(x,t)$ admits the following Fourier-Legendre series representation
\begin{equation}
G(x,t)=\sum_{n=0}^{\infty} \frac{g_{n}(x)}{x}P_{2n}\left(  \frac{t}{x}\right)
,\quad0<t\leq x\leq\pi, \label{G Fourier series}%
\end{equation}
where $P_{k}$ stands for the Legendre polynomial of order $k$.

For every $x\in(0,\pi]$ the series converges in the norm of $L_{2}(0,x)$. The
first coefficient $g_{0}(x)$ has the form%
\begin{equation}
g_{0}(x)=\varphi(0,x)-1, \label{g0}%
\end{equation}
and the rest of the coefficients can be calculated following a simple
recurrent integration procedure.
\end{theorem}

Notice that from (\ref{Gxx}) and (\ref{G Fourier series}) the equality follows%
\begin{equation}
G(x,x)=\sum_{n=0}^{\infty} \frac{g_{n}(x)}{x}. \label{Gxx series}%
\end{equation}

\begin{remark}
\label{Rem q from g0}Equality \eqref{g0} shows that the potential $q(x)$ can
be recovered from $g_{0}(x)$. Indeed,
\begin{equation}
q(x)=\frac{\varphi^{\prime\prime}(0,x)}{\varphi(0,x)}=\frac{g_{0}%
^{\prime\prime}(x)}{g_{0}(x)+1}.\label{q from g0}%
\end{equation}
Moreover, the constant $h$ is also recovered directly from $g_{0}(x)$, since%
\[
h=g_{0}^{\prime}(0).
\]

\end{remark}

\subsection{Neumann series of Bessel functions representations for solutions
and their derivatives}

With the aid of Theorem \ref{Theorem G series representation} the following
series representations for the solution $\varphi(\rho,x)$ and for its
derivative are obtained.

\begin{theorem}
[\cite{KNT}]\label{Theorem NSBF} The solution $\varphi(\rho,x)$ and its
derivative with respect to $x$ admit the following series representations%
\begin{equation}
\varphi(\rho,x)=\cos\rho x+\sum_{n=0}^{\infty}(-1)^{n}g_{n}(x)j_{2n}(\rho x)
\label{phi}%
\end{equation}
and
\begin{equation}
\varphi^{\prime}(\rho,x)=-\rho\sin\rho x+\left(  h+\frac{1}{2}\int_{0}%
^{x}q(s)\,ds\right)  \cos\rho x+\sum_{n=0}^{\infty}(-1)^{n}\gamma_{n}%
(x)j_{2n}(\rho x), \label{phiprime}%
\end{equation}
where $j_{k}(z)$ stands for the spherical Bessel function of order $k$,
defined as $j_{k}(z):=\sqrt{\frac{\pi}{2z}}J_{k+\frac{1}{2}}(z)$ (and $J_{l}$
stands for the Bessel function of order $l$) (see, e.g.,
\cite{AbramowitzStegunSpF}). The coefficients $g_{n}(x)$ are those from
Theorem \ref{Theorem G series representation}, and the coefficients
$\gamma_{n}(x)$ can be calculated following a simple recurrent integration
procedure, starting with
\begin{equation}
\gamma_{0}(x)=\varphi^{\prime}(0,x)-h-\frac{1}{2}\int_{0}^{x}q(s)\,ds.
\label{gamma0}%
\end{equation}
For every $\rho\in\mathbb{C}$ both series converge pointwise. For every
$x\in\left[  0,\pi\right]  $ the series converge uniformly on any compact set
of the complex plane of the variable $\rho$, and the remainders of their
partial sums admit estimates independent of $\operatorname{Re}\rho$.
\end{theorem}

We refer the reader to \cite{KNT} and \cite{KrBook2020} for the proof and
additional details related to this result. The coefficients $g_{n}$ and
$\gamma_{n}$ decay as $n\to\infty$, and the decay rate is faster for smoother
potentials. Namely, the following result is valid.

\begin{proposition}
[\cite{KrT2018}]\label{Prop coeff decay} Let $q\in W_{\infty}^{p}[0,\pi]$ for
some $p\in\mathbb{N}_{0}$, i.e., the potential $q$ is $p$ times differentiable
with the last derivative being bounded on $[0,\pi]$. Then there exist
constants $c$ and $d$, independent of $N$, such that
\[
|g_{N}(x)|\leq\frac{cx^{p+2}}{(2N-1)^{p+1/2}}\qquad\text{and}\qquad|\gamma
_{N}(x)|\leq\frac{dx^{p+1}}{(2N-1)^{p-1/2}},\qquad2N\geq p+1.
\]

\end{proposition}

\section{Method of solution of Problem \ref{Problem 1}}

The idea of the method was proposed in \cite{Kr2019JIIP}, and its modification
involving the integrated Gelfand-Levitan equation
(\ref{Gelfand-Levitan integrated}) instead of (\ref{Gelfand-Levitan}) was
developed in \cite[Section 13.4]{KrBook2020}. We begin by briefly recalling
that result.

From now on and without loss of generality, we suppose that $\rho_{0}=0$. This
always can be achieved by a simple shift of the potential. If originally
$\rho_{0}\neq0$, then we can consider the new potential $\widetilde
{q}(x):=q(x)-\rho_{0}^{2}$. Obviously, the eigenvalues $\left\{  \lambda
_{n}\right\}  _{n=0}^{\infty}$ and $\left\{  \nu_{n}\right\}  _{n=0}^{\infty}$
shift by the same amount, while the numbers $h$, $H$ and $\left\{  \alpha
_{n}\right\}  _{n=0}^{\infty}$ do not change. After recovering $\widetilde
{q}(x)$ from the shifted eigenvalues, one gets the original potential $q(x)$
by adding $\rho_{0}^{2}$ back.

\subsection{Main system of linear algebraic equations}

Denote%
\begin{align}
C_{km}(x)  &  =
\begin{cases}
c_{k,k}(x)+\frac{1}{(4k+1)(4k+3)}, & \text{when }m=k,\\
c_{k,k+1}(x)-\frac{1}{(4k+3)(4k+5)}, & \text{when }m=k+1,\\
c_{k,m}(x), & \text{otherwise},
\end{cases}
\label{C_km}\\
c_{0,m}(x)  &  = \left(  \frac{1}{\alpha_{0}}-\frac{1}{\pi}\right)
\frac{x\delta_{0m}}{3}+\left(  -1\right)  ^{m}\sum_{n=1}^{\infty}\left(
\frac{j_{2m}(\rho_{n}x)j_{1}(\rho_{n}x)}{\alpha_{n}\rho_{n}}-\frac
{2j_{2m}(nx)j_{1}(nx)}{\pi n}\right)  ,\label{c_0m}\\
c_{k,m}(x)  &  =\left(  -1\right)  ^{m+k}\sum_{n=1}^{\infty}\left(
\frac{j_{2m}(\rho_{n}x)j_{2k+1}(\rho_{n}x)}{\alpha_{n}\rho_{n}}-\frac
{2j_{2m}(nx)j_{2k+1}(nx)}{\pi n}\right)  ,\qquad k=1,2,\ldots, \label{c_km}%
\end{align}
and
\begin{equation}
d_{k}(x)= \left(  \frac{1}{\pi}-\frac{1}{\alpha_{0}}\right)  \frac
{x\delta_{0k}}{3}-\left(  -1\right)  ^{k}\sum_{n=1}^{\infty}\left(  \frac
{\cos\rho_{n}x\,j_{2k+1}(\rho_{n}x)}{\alpha_{n}\rho_{n}}-\frac{2\cos
nx\,j_{2k+1}(nx)}{\pi n}\right),  \label{d_k}%
\end{equation}
where $\delta_{0m}$ stands for the Kronecker delta. The series in these
definitions converge uniformly with respect to $x\in\left[  0,\pi\right]  $.
This follows from (\ref{asympt 1}) and asymptotics of spherical Bessel
functions
\[
j_{k}(t)\sim\frac{\sin(t-\frac{k\pi}{2})}{t}+O\left(  \frac{1}{t^{2}}\right)
,\quad\text{when }t\in\mathbb{R}\text{ and }t\rightarrow\infty.
\]

\begin{theorem}
[\cite{KrBook2020}]\label{Thm main system} The coefficients $g_{m}(x)$ satisfy
the system of linear algebraic equations%
\begin{equation}
\sum_{m=0}^{\infty}g_{m}(x)C_{km}(x)=d_{k}(x),\qquad\text{for all
}k=0,1,\ldots.\label{G-L-i}%
\end{equation}
For all $x\in\left[  0,\pi\right]  $ and $k=0,1,\ldots$ the series in
\eqref{G-L-i} converges.
\end{theorem}

It is of crucial importance the fact that for recovering the potential $q$ as
well as the constant $h$ it is not necessary to compute many coefficients
$g_{m}(x)$ (that would be equivalent to approximate solution of the integrated
Gelfand-Levitan equation) but instead the sole $g_{0}(x)$ is sufficient for
this purpose, see Remark \ref{Rem q from g0} as well as Example \ref{Ex2}
where we illustrate with numerical results the advantage of recovering $q$
from $g_{0}$ and not from the kernel $G$. The system (\ref{G-L-i}) is obtained
by using the integrated Gelfand-Levitan equation
(\ref{Gelfand-Levitan integrated}). As we show below, the results are even
better and the stability of the method can be proved if the system of linear
algebraic equations for the coefficients $g_{m}(x)$ is obtained from the
Gelfand-Levitan equation (\ref{Gelfand-Levitan}) where the representation
(\ref{F alt}) for the input Gelfand-Levitan kernel $F(x,t)$ is used. Since the
representation (\ref{F alt}) requires the knowledge of the parameter $\omega$
we show first how it can be computed.

\subsection{Recovery of $\omega$ and of the characteristic function }

\label{subsect recovery hn}

For all $k=0,1,\ldots$ from (\ref{bc1}) and Theorem \ref{Theorem NSBF} we have
that%
\begin{equation}
\varphi^{\prime}(\rho_{k},\pi)+H\varphi(\rho_{k},\pi)=-\rho_{k}\sin\rho_{k}%
\pi+\omega\cos\rho_{k}\pi+\sum_{n=0}^{\infty}(-1)^{n}\bigl(\gamma_{n}%
(\pi)+Hg_{n}(\pi)\bigr)j_{2n}(\rho_{k}\pi)=0.\label{first system}%
\end{equation}
Let us denote
\[
h_{n}:=\gamma_{n}(\pi)+Hg_{n}(\pi),\qquad n\geq0.
\]

Later on we show that the coefficients $h_{n}$ are necessary components for
the solution of Problem \ref{Problem 2}. Also they can be used to perform the
\textquotedblleft flipping of the potential\textquotedblright, i.e., to
transform Problem 1 into the inverse problem for the potential $q(\pi-x)$
having the parameters $h$ and $H$ interchanged, see Subsection
\ref{Subsect reverse interval} for details.

Notice that due to the supposition $\rho_{0}=0$, we have that
\[
\varphi^{\prime}(0,\pi)+H\varphi(0,\pi)=0.
\]
From Theorem \ref{Theorem NSBF} it follows that
\[
\varphi^{\prime}(0,\pi)+H\varphi(0,\pi)=\gamma_{0}(\pi)+\omega_{1}+H(g_{0}%
(\pi)+1) = \gamma_{0}(\pi)+Hg_{0}(\pi) + \omega_{1}+H = h_{0} + \omega.
\]
Thus,
\begin{equation}
\omega= -h_{0}. \label{omega = -h0}%
\end{equation}
This equality coincides with (\ref{first system}) for $k=0$, while for
$k=1,2,\ldots$ equality (\ref{first system}) can be written in the form
\begin{equation}
h_{0}\cos\rho_{k}\pi-\sum_{n=0}^{\infty}(-1)^{n}h_{n}j_{2n}(\rho_{k}\pi
)=-\rho_{k}\sin\rho_{k}\pi,\quad k=1,2,\ldots. \label{syst for hn}%
\end{equation}

Hence the numbers $h_{n}$, including $h_{0}=-\omega$, can be approximated by
solving the system%
\begin{equation}
h_{0}\cos\rho_{k}\pi-\sum_{n=0}^{N_{1}-1}(-1)^{n}h_{n}j_{2n}(\rho_{k}\pi
)=-\rho_{k}\sin\rho_{k}\pi,\quad k=1,\ldots,N_{1}.\label{syst1}%
\end{equation}

Notice that the recovery of the coefficients $h_{n}$ leads to the recovery of
the characteristic function of the Sturm-Liouville problem (\ref{SL equation}%
), (\ref{bc1}), $\Phi(\rho)=\varphi^{\prime}(\rho,\pi)+H\varphi(\rho,\pi)$ for
all $\rho$:%

\[
\Phi(\rho)=-\rho\sin\rho\pi+\omega\cos\rho\pi+\sum_{n=0}^{\infty}(-1)^{n}%
h_{n}j_{2n}(\rho\pi),
\]
and for its calculation the knowledge of $H$ is superfluous.

\begin{remark}\label{Rmk Alternative omega}
Solution of the system \eqref{syst1} leads to the recovery of the parameter
$\omega$, due to \eqref{omega = -h0}. Notice that this way of obtaining
$\omega$ does not involve the asymptotics of the eigenvalues.
\end{remark}

The system \eqref{syst1} can be ill-conditioned for large values of $N_{1}$
(say, for more than 30 eigenvalues). Of course,\ when solving an inverse
spectral problem one expects to make use of all known eigenvalues. Thus
discarding some of them is not an option. To avoid the problem of the
ill-conditioning, one can look for less coefficients $h_{n}$ than the number
of the given eigenvalues $\rho_{k}$, i.e., to consider an overdetermined
system
\begin{equation}
h_{0}\cos\rho_{k}\pi-\sum_{n=0}^{N_{1}^{o}}(-1)^{n}h_{n}j_{2n}(\rho_{k}%
\pi)=-\rho_{k}\sin\rho_{k}\pi,\quad k=1,\ldots,N_{1} \label{syst1 overdet}%
\end{equation}
choosing for the number of coefficients $N_{1}^{o}+1$ such value that the
condition number of the coefficient matrix remains relatively small. The
system can be solved using the Moore-Penrose pseudoinverse. Note that the
exact coefficients $h_{n}$ decay rather fast (see Proposition
\ref{Prop coeff decay}), so taking less approximate coefficients does not
present any problem.

\begin{remark}
\label{Rmk alternate hn} As it follows from \eqref{bc1} and
\eqref{sol transmutation},
\begin{equation}
\label{moments problem1}0=\varphi^{\prime}(\rho_{k},\pi)+H\varphi(\rho_{k}%
,\pi) = -\rho_{k}\sin\rho_{k}\pi+ \omega\cos\rho_{k}\pi+ \int_{0}^{\pi
}\bigl(G_{x}^{\prime}(\pi,t) + HG(\pi,t)\bigr)\cos\rho_{k}t\,dt.
\end{equation}

Let $G:[-\pi,\pi]\rightarrow\mathbb{R}$ be an even function such that
$G(t)=G_{x}^{\prime}(\pi,t)+HG(\pi,t)$ for $t\in\lbrack0,\pi]$. Then
\eqref{moments problem1} can be written as
\begin{equation}
\int_{-\pi}^{\pi}G(t)e^{i\rho_{k}t}\,dt=2\rho_{k}\sin\rho_{k}\pi-2\omega
\cos\rho_{k}\pi, \label{moments problem2}%
\end{equation}
and equality \eqref{moments problem2} holds for all values $\pm\rho_{k}$,
$k\geq0$. Such problem is known as a non-harmonic moments problem
\cite{AvdoninIvanov}, \cite{Russell1967}, \cite{Young2001}. Searching the
solution of the problem \eqref{moments problem2} using the basis of Legendre
polynomials results in the system \eqref{omega = -h0}, \eqref{syst for hn}.

One can think of another way of solving the problem \eqref{moments problem2}
using the fact that the system $\{1,t\}\cup\{e^{\pm i\rho_{k}t}\}_{k=1}%
^{\infty}$ is a Riesz basis in $L_{2}(-\pi,\pi)$. So the function $G$ can be
approximately recovered by the least squares method, see \cite[\S 5]{Mihlin}.
Taking into account the parity of the function $G$, the minimization problem
reduces to $\min_{a_{0},\ldots,a_{n}}\left\Vert G(t)-\sum_{k=0}^{n}a_{k}%
\cos\rho_{k}t\right\Vert _{L_{2}(0,\pi)}$, and its solution can be found from
the linear system
\begin{equation}
\sum_{k=0}^{n}a_{k}\left(  \cos\rho_{k}t,\cos\rho_{m}t\right)  _{L_{2}(0,\pi
)}=\left(  G(t),\cos\rho_{m}t\right)  _{L_{2}(0,\pi)}=\rho_{m}\sin\rho_{m}%
\pi-\omega\cos\rho_{m}\pi,\quad m=0,\ldots,n. \label{moments problem approx}%
\end{equation}
The system \eqref{moments problem approx} is numerically stable, possessing\ a
small and uniformly bounded with respect to $n$ condition number, and its
solution provides an approximation $G_{n}(t)=\sum_{k=0}^{n}a_{k}\cos\rho_{k}t$
known to converge to the exact function $G(t)$ in the $L_{2}(0,\pi)$ space. A
similar idea was implemented in \cite{Rundell Sacks}. Once the coefficients
$a_{k}$ are found, one can obtain approximate values of the coefficients
$h_{m}$ as well,
\[%
\begin{split}
h_{m}  &  =(4m+1)\int_{0}^{\pi}G(t)P_{2m}\left(  \frac{t}{\pi}\right)
\,dt\approx(4m+1)\sum_{k=0}^{n}a_{k}\int_{0}^{\pi}P_{2m}\left(  \frac{t}{\pi
}\right)  \cos\rho_{k}t\,dt\\
&  =(-1)^{m}(4m+1)\pi\sum_{k=0}^{n}a_{k}j_{2m}(\rho_{k}\pi),\qquad
m=0,1,\ldots
\end{split}
\]

Even though such scheme is numerically stable and approximate coefficients
$h_{m}$ are guaranteed to converge to the exact ones, the convergence is slow
and in practice the approach based directly on the system
\eqref{syst1 overdet} produces more accurate results. The main reason is that
the coefficients $a_{k}$ decay slowly, as $1/k^{2}$, resulting in a large
error in the approximate coefficients $h_{m}$.
\end{remark}

In the next section we develop an additional possibility for recovering the
parameter $\omega$ based on a more standard approach assuming the knowledge of
a sufficient amount of spectral data allowing an efficient use of their
asymptotics. It can be applied for calculating together with  $\omega$ some
other asymptotic parameters of the spectral data.

\subsection{Recovery of the parameter $\omega$ and coefficients of the
asymptotic expansions of eigenvalues $\rho_{n}$ and norming constants
$\alpha_{n}$}

\label{Subsect recovery omega} As it is well known (see, e.g., \cite{Rundell
Sacks}), the parameter $\omega$, in principle, can be recovered using the
asymptotics of the eigenvalues. Indeed, it immediately follows from
\eqref{asympt 1} that
\[
\omega=\pi\lim_{n\rightarrow\infty}n(\rho_{n}-n).
\]
Note that due to \eqref{asympt 1} the sequence $\left\{  n(\rho_{n}%
-n)-\frac{\omega}{\pi}\right\}  _{n=0}^{\infty}=\{k_{n}\}_{n=0}^{\infty}$
belongs to $\ell_{2}$. Suppose that only a finite set of eigenvalues $\rho
_{0},\ldots,\rho_{N_{1}}$ is given. Then an approximate value of the parameter
$\omega$ can be found by minimizing the $\ell_{2}$ norm of the sequence
$\left\{  n(\rho_{n}-n)-\frac{\omega}{\pi}\right\}  _{n=N_{s}}^{N_{1}}$. Here
$N_{s}$ is some parameter between $1$ and $N_{1}$ chosen to skip several first
eigenvalues which can differ a lot from the asymptotic formula. One can take,
e.g., $N_{s}=[N_{1}/2]$, and obtain that
\begin{equation*}
\omega\approx\operatorname*{arg\,min}_{\omega}\sum_{n=[N_{1}/2]}^{N_{1}}\left(
n(\rho_{n}-n)-\frac{\omega}{\pi}\right)  ^{2}.
\end{equation*}

Better accuracy can be achieved though by taking into account higher order
asymptotic terms for the eigenvalues. Recall \cite[Remark 1.1.1]{Yurko2007}
that for $q\in W_{2}^{N}$, $N\geq1$, one has
\begin{equation}
\rho_{n}=n+\sum_{j=1}^{N+1}\frac{\omega^{(j)}}{n^{j}}+\frac{\kappa_{n}%
}{n^{N+1}}\qquad\text{and}\qquad\alpha_{n}=\frac{\pi}{2}+\sum_{j=1}^{N+1}%
\frac{\omega^{(j,+)}}{n^{j}}+\frac{\chi_{n}}{n^{N+1}},\label{rn refined}%
\end{equation}
where $\omega^{(1)}=\frac{\omega}{\pi}$, $\omega^{(2k)}=0$, $\omega
^{(2k-1,+)}=0$, $k>0$ and $\{\kappa_{n}\}_{n=0}^{\infty}\in\ell_{2}$,
$\{\chi_{n}\}_{n=0}^{\infty}\in\ell_{2}$. The first eigenvalues and norming
constants can be rather distant from the asymptotic formulas (obtained by
discarding the remainder terms $\kappa_{n}/n^{N+1}$ and $\chi_{n}/n^{N+1}$),
but higher index eigenvalues and norming constants are quite close. So one can
recover the parameter $\omega$ by finding approximate coefficients
$\omega^{(2j-1)}$, $1\leq j\leq K$ from the best fit problem
\begin{equation}
\sum_{j=1}^{K}\frac{\omega^{(2j-1)}}{n^{2j-1}}=\rho_{n}-n,\qquad N_{s}\leq
n\leq N_{1}.\label{best fit for rhon}%
\end{equation}
Since the smoothness of the potential is unknown beforehand, the parameter
$K<N_{1}-N_{s}$ is chosen such that the least squares error of the fit for $K$
is significantly better than the least squares error for $K-1$. In practice,
depending on the potential, an optimal value results to be 2 or 3. For the
parameter $N_{s}$, again, one can choose $[N_{1}/2]$.

Suppose that also a finite set of norming constants $\alpha_{0},\ldots
,\alpha_{N_{1}}$ is given. Similarly, the approximate coefficients
$\omega^{(2j,+)}$, $1\leq j\leq K-1$ can be found from the best fit problem
\begin{equation}
\sum_{j=1}^{K-1}\frac{\omega^{(2j,+)}}{n^{2j}}=\alpha_{n}-\frac{\pi}{2},\qquad
N_{s}\leq n\leq N_{1}. \label{best fit for alphan}%
\end{equation}

\begin{remark}
Taking into account that the residual sequences $\{\kappa_{n}\}_{n=0}^{\infty
}$ and $\{\chi_{n}\}_{n=0}^{\infty}$ belong to the $\ell_{2}$ space, one can
obtain approximate asymptotic coefficients $\omega^{(2j-1)}$, $1\leq j\leq K$
and $\omega^{(2j,+)}$, $1\leq j\leq K-1$ by minimizing the $\ell_{2}$ norms of
the tails of these sequences, i.e., from the least squares problems
\[
\min_{\omega^{(1)},\ldots,\omega^{(2K-1)}}\sum_{n=N_{s}}^{N_{1}}%
n^{4K-2}\biggl[\rho_{n}-n-\sum_{j=1}^{K}\frac{\omega^{(2j-1)}}{n^{2j-1}%
}\biggr]^{2}%
\]
and
\[
\min_{\omega^{(2,+)},\ldots,\omega^{(2K-2,+)}}\sum_{n=N_{s}}^{N_{1}}%
n^{4K-2}\biggl[\alpha_{n}-\frac{\pi}{2}-\sum_{j=1}^{K-1}\frac{\omega^{(2j,+)}%
}{n^{2j}}\biggr]^{2}.
\]
However, the numerical performance of the method based on
\eqref{best fit for rhon} and \eqref{best fit for alphan} was slightly better.
\end{remark}

\begin{remark}
\label{Remark add asymptotic} The coefficients $\omega^{(2j-1)}$, $1\leq j\leq
K$ and $\omega^{(2j,+)}$, $1\leq j\leq K-1$, if computed on the first step, as
described above, are not required but turn out useful when computing the
coefficients of the system \eqref{G-L-i}. Indeed, suppose we are given a
finite set of the spectral data $\left\{  \rho_{n},\,\alpha_{n}\right\}
_{n=0}^{N_{1}}$. Truncating the series \eqref{C_km}--\eqref{d_k} at $n=N_{1}$
is equivalent to taking the missing eigenvalues and norming constants to be
$\rho_{n}=n$ and $\alpha_{n}=\pi/2$, $n>N_{1}$.

In order to compute the coefficients $C_{km}(x)$ and $d_{k}(x)$ more
accurately one can add to this set an arbitrarily large number of the
\textquotedblleft asymptotic spectral data\textquotedblright\
\[
\biggl\{\rho_{n}=n+\sum_{j=1}^{K}\frac{\omega^{(2j-1)}}{n^{2j-1}},\,\alpha
_{n}=\frac{\pi}{2}+\sum_{j=1}^{K-1}\frac{\omega^{(2j,+)}}{n^{2j}%
}\biggr\}_{n=N_{1}+1}^{M},
\]
see \eqref{rn refined}. The value of $M$ can be chosen very large.
\end{remark}

Now, disposing of two alternative procedures for computing the parameter
$\omega$ we are in a position to obtain  the new system of linear algebraic
equations for the coefficients $g_{n}(x)$ and prove the stability of the
numerical method based on it.

\subsection{Alternative main system of linear algebraic equations}

\label{Subsect Alternative system} A system similar to \eqref{G-L-i} can be
obtained using the representation \eqref{F alt} instead of \eqref{Ftilde}. An
advantage of \eqref{F alt} is in preserving the Fredholm form of the
Gelfand-Levitan equation and at the same time resolving the problem of slow
convergence of the series \eqref{F} near $x=t=\pi$. Moreover, the structure of
the representation \eqref{F alt} extends further the idea discussed in Remark
\ref{Remark add asymptotic} by taking into account the whole infinite series
of \textquotedblleft asymptotic spectral data\textquotedblright\ and not just
some (large) number of terms. In \eqref{F alt} only the first order
approximations are used, so still some number of \textquotedblleft asymptotic
spectral data\textquotedblright\ may be added explicitly. Nevertheless, the
proposed modification based on \eqref{F alt} allows one to save the
computational cost of dealing with thousands of terms in
\eqref{C_km}--\eqref{d_k} which are necessary to achieve a comparable accuracy
when using the truncated system \eqref{G-L-i}. And working with the original
Gelfand-Levitan equation allows us to prove the stability results for the
approximate solutions obtained by truncating the infinite system.

Let us denote
\begin{align}
\widetilde c_{km}(x)  &  = -\frac{\omega x}{8\pi} \left(  \frac{\delta
_{m,k-1}}{(2k-3/2)_{3}} - \frac{2\delta_{m,k}}{(2k-1/2)_{3}} + \frac
{\delta_{m,k+1}}{(2k+1/2)_{3}}\right) \nonumber\\
&  \quad+ (-1)^{k+m}\sum_{n=1}^{\infty}\left[  \frac{j_{2k}(\rho_{n} x)
j_{2m}(\rho_{n} x)}{\alpha_{n}} - \frac{2 j_{2k}(nx) j_{2m}(nx)}{\pi} \right.
\label{c_km tilde}\\
&  \quad+ \left.  \frac{2\omega}{\pi^{2} n}\left(  x j_{2k}(nx) j_{2m+1}(nx) +
x j_{2k+1}(nx)j_{2m}(nx)- \frac{2(k+m)j_{2k}(nx) j_{2m}(nx)}{n}\right)
\right]  ,\nonumber\\
\widetilde C_{0m}(x)  &  = \left(  \frac1{\alpha_{0}} - \frac1\pi
+\frac{2\omega x^{2}}{3\pi^{2}}\right)  \delta_{0m} + \frac{2\omega x^{2}%
}{15\pi^{2}}\delta_{1m} + \widetilde c_{0m}(x),\label{C_0m tilde}\\
\widetilde C_{1m}(x)  &  = \frac{2\omega x^{2}}{15\pi^{2}}\delta_{0m} +
\widetilde c_{1m}(x),\label{C_1m tilde}\\
\widetilde C_{km}(x)  &  = \widetilde c_{km}(x),\qquad k=2,3,\ldots,
\ m\in\mathbb{N}_{0}, \label{C_km tilde}%
\end{align}
and
\begin{equation}
\label{d_k til}%
\begin{split}
\widetilde d_{k}(x)  &  = -\left(  \frac1{\alpha_{0}} - \frac1\pi
+\frac{4\omega x^{2}}{3\pi^{2}}-\frac{\omega x}{\pi}\right)  \delta_{k0} -
\frac{2\omega x^{2}}{15\pi^{2}}\delta_{k1}\\
&  \quad-(-1)^{k}\sum_{n=1}^{\infty}\left[  \frac{\cos\rho_{n}x}{\alpha_{n}%
}j_{2k}(\rho_{n}x)-\frac{2\cos nx}{\pi}j_{2k}(nx)\right. \\
&  \quad+ \left.  \frac{2\omega}{\pi^{2} n} \left(  x \sin nx j_{2k}(nx) +
x\cos nx j_{2k+1}(nx) - \frac{2k}{n} \cos nx j_{2k}(nx)\right)  \right]  ,
\end{split}
\end{equation}
where $\delta_{k,m}$ stands for the Kronecker delta and $(k)_{m}$ is the
Pochhammer symbol.

Then the following result is valid.

\begin{theorem}
\label{Thm main system alt} The coefficients $g_{m}(x)$ satisfy the system of
linear algebraic equations
\begin{equation}
\frac{g_{k}(x)}{(4k+1)x} + \sum_{m=0}^{\infty}g_{m}(x)\widetilde
C_{km}(x)=\widetilde d_{k}(x),\qquad k=0,1,\ldots. \label{G-L-alt}%
\end{equation}
For all $x\in\left[  0,\pi\right]  $ and $k=0,1,\ldots$ the series in
\eqref{G-L-alt} converges.
\end{theorem}

\begin{proof}
The idea is to substitute expressions (\ref{F alt}) and
(\ref{G Fourier series}) into (\ref{Gelfand-Levitan}). Consider
first
\begin{equation}\label{Integral F alt}
\begin{split}
\int_{0}^{x} P_{2m}\left(  \frac{s}{x}\right) & F(s,t)\frac{ds}{x} =\left(
\frac{1}{\alpha_{0}}-\frac{1}{\pi}+\frac{\omega t^2}{\pi^2}\right)  \int_{0}^{x}
P_{2m}\left(  \frac{s}{x}\right)  \frac{ds}{x} \\
& \quad + \frac{\omega}{\pi^2}\int_0^x s^2 P_{2m}\left(  \frac{s}{x}\right) \frac{ds}{x} -\frac{\omega}{\pi}\int_0^x  P_{2m}\left(  \frac{s}{x}\right)  \max\{s,t\} \frac{ds}{x}\\
&\quad +\sum_{n=1}^{\infty}\left[  \frac{\cos\rho_{n}t}{\alpha_{n}}\int
_{0}^{x}P_{2m}\left(  \frac{s}{x}\right)  \cos\rho_{n}s\frac{ds}{x} -\frac{\cos
nt}{\alpha_{n}^{0}}\int_{0}^{x}P_{2m}\left(  \frac{s}{x}\right)  \cos
ns\frac{ds}{x}\right.\\
& \quad+\left.\frac{2\omega}{\pi^2n}\left( \cos nt \int_0^x  P_{2m}\left(  \frac{s}{x}\right) \cdot s\sin ns  \frac{ds}{x} + t\sin nt \int_0^x  P_{2m}\left(  \frac{s}{x}\right)\cos ns  \frac{ds}{x}\right) \right].
\end{split}
\end{equation}
The change of the sum and of the integration is possible due to uniform convergence of the series for $F$.
Let us calculate the integrals appearing in \eqref{Integral F alt}. Using the orthogonality of the Legendre polynomials we obtain
\begin{align*}
\int_0^x P_{2m}\left(\frac sx\right)\,\frac{ds}x &= \delta_{0m}, \\
\int_0^x s^2 P_{2m}\left(\frac sx\right)\,\frac{ds}x &= x^2 \int_0^1 \tau^2 P_{2m}(\tau)\,d\tau =
\begin{cases}
x^2/3, & m=0,\\
2x^2/15, & m=1,\\
0, & m>1.
\end{cases}
\end{align*}
Recall that
\[
\int_0^x P_{2m}\left(\frac sx\right)\cos bs\,\frac{ds}x = (-1)^m j_{2m}(bx).
\]
Differentiating both sides with respect to $b$ we obtain that
\begin{equation*}
\int_0^x P_{2m}\left(\frac sx\right) s\sin bs\,\frac{ds}x = (-1)^m \left(x j_{2m+1}(bx) - \frac {2m}{b} j_{2m}(bx)\right).
\end{equation*}
Now,
\begin{equation*}
\int_0^x  \max\{s,t\} P_{2m}\left(\frac sx\right)\frac{ds}x = t\int_0^t P_{2m}\left(\frac sx\right)\frac{ds}x + \int_t^x \frac sx P_{2m}\left(\frac sx\right)\,ds.
\end{equation*}
Let $m>0$. Using the identities (note that $P_{2n\pm 1}(0)=0$)
\begin{equation*}
\int_{0}^{t}P_{2n}\left(  \frac{\tau}{x}\right)  d\tau=\frac{x}{4n+1}\left(
P_{2n+1}\left(  \frac{t}{x}\right)  -P_{2n-1}\left(  \frac{t}{x}\right)
\right)  ,\qquad n=1,2,\ldots
\end{equation*}
and
\begin{equation*}
xP_n(x) = \frac{n+1}{2n+1} P_{n+1}(x) + \frac{n}{2n+1} P_{n-1}(x)
\end{equation*}
we obtain that
\begin{align*}
t\int_0^t  P_{2m}\left(\frac sx\right)\frac{ds}x &= \frac{t}{4m+1}\left(P_{2m+1}\left(\frac tx\right)-P_{2m-1}\left(\frac tx\right)\right) \\
& = \frac{(2m+2)x P_{2m+2}(t/x)}{(4m+1)(4m+3)} - \frac{xP_{2m}(t/x)}{(4m-1)(4m+3)} - \frac{(2m-1)xP_{2m-2}(t/x)}{(4m-1)(4m+1)}.
\end{align*}
Similarly,
\begin{align*}
\int_t^x \frac sx P_{2m}\left(\frac sx\right)\,ds &= \int_t^x \left(\frac{2m+1}{4m+1} P_{2m+1}\left(\frac sx\right)+\frac{2m}{4m+1}P_{2m-1}\left(\frac sx\right)\right)\,ds \\
&= \frac{(2m+1)x}{(4m+1)(4m+3)}\left(P_{2m+2}\left(\frac sx\right) - P_{2m}\left(\frac sx\right)\right)\Big|_t^x \\
&\quad +
\frac{2m x}{(4m-1)(4m+1)}\left(P_{2m}\left(\frac sx\right) - P_{2m-2}\left(\frac sx\right)\right)\Big|_t^x\\
&= -\frac{(2m+1)xP_{2m+2}(t/x)}{(4m+1)(4m+3)} - \frac{xP_{2m}(t/x)}{(4m-1)(4m+3)}+\frac{2mxP_{2m-2}(t/x)}{(4m-1)(4m+1)}.
\end{align*}
Hence,
\begin{equation}\label{Int05}
\int_0^x  \max\{s,t\} P_{2m}\left(\frac sx\right)\frac{ds}x = \frac{x P_{2m+2}(t/x)}{(4m+1)(4m+3)} - \frac{2x P_{2m}(t/x)}{(4m-1)(4m+3)} + \frac{x P_{2m-2}(t/x)}{(4m-1)(4m+1)}.
\end{equation}
One can easily check that equality \eqref{Int05} is valid also for $m=0$ if one takes $P_{-2}\equiv 0$.
Thus, the Gelfand-Levitan equation
(\ref{Gelfand-Levitan}) takes the form
\begin{align*}
-F(x,t) & = \sum_{n=0}^{\infty}\frac{g_{n}(x)}x P_{2n}\left(  \frac{t}{x}\right) +\sum_{m=0}^{\infty}g_{m}(x)\left\{\left(  \frac{1}{\alpha_{0}}-\frac{1}{\pi} +\frac{\omega x^2}{3\pi^2}+t^2\right)  \delta_{0m}+\frac{2\omega x^2}{15\pi^2}\delta_{m1}\right.\\
&\quad -\frac{\omega x}{\pi} \left( \frac{ P_{2m+2}(t/x)}{(4m+1)(4m+3)} - \frac{2P_{2m}(t/x)}{(4m-1)(4m+3)} + \frac{P_{2m-2}(t/x)}{(4m-1)(4m+1)}\right)\\
&\quad +
\left(  -1\right)  ^{m}\sum_{n=1}^{\infty}\left[\frac
{\cos\rho_{n}t}{\alpha_{n}}j_{2m}(\rho_{n}x)-\frac{2\cos nt}{\pi
n}j_{2m}(nx)\right.\\
& \quad + \left.\left.\frac{2\omega}{\pi^2 n} \left( x\cos nt j_{2m+1}(nx) + t\sin nt j_{2m}(nx) - \frac{2m}n \cos nt j_{2m}(nx)\right)\right]\right\}.
\end{align*}
Let us multiply this equality by $\frac 1x P_{2k}\left(  \frac{t}{x}\right)  $ and
integrate with respect to $t$ from $0$ to $x$. Note that now $\max\{x,t\}=x$, hence
\[
\int_0^x \max\{x,t\}P_{2k}\left(\frac tx\right)\,\frac{dt}{x} = \int_0^x P_{2k}\left(\frac tx\right) \,dt = x\delta_{0k},
\]
and one easily obtains that $\int_0^x F(x,t) P_{2k}\left(\frac tx\right)\, \frac{dt}x  = -\widetilde d_k(x)$, while the right-hand side reduces to $\frac{g_m(x)}{(4k+1)x}+\sum_{m=0}^{\infty}g_{m}(x)\widetilde C_{km}(x)$.
In order to verify the convergence of the series in (\ref{G-L-alt}) it is
sufficient to notice that for every $x\in\left(  0,\pi\right]  $ the series
$\sum_{m=0}^{\infty}g_{m}(x)\tilde C_{k,m}(x)$ is the inner product of two $\ell_{2}%
$-sequences. Indeed, note that the set of functions
\begin{equation}\label{basis pm}
\left\{p_m(t)\right\}_{m=0}^\infty :=\left\{\frac{\sqrt{4m+1} P_{2m}(t/x)}{\sqrt x}\right\}_{m=0}^\infty
\end{equation}
forms an orthonormal basis in $L_2(0,x)$. Hence the numbers $\frac{g_{m}(x)}{\sqrt x\sqrt{4m+1}}$ are the Fourier coefficients of the function
$G(x,t)$ with respect to the basis $\left\{  p_m(t)  \right\}  _{m=0}^{\infty}$, while the numbers
\[
\begin{split}
\sqrt x \sqrt{4m+1} \widetilde C_{k,m}(x) & = \sqrt x \sqrt{4m+1}\int_0^x \int_0^x F(s,t) P_{2m}\left(\frac sx\right) P_{2k}\left(\frac tx\right) \frac{ds}x \frac{dt}x \\
& = \int_0^x \biggl(\int_0^1 F(s,x\tau) P_{2k}(\tau)d\tau\biggr) p_m(s)\,ds
\end{split}
\]
are the Fourier coefficients of the function $F_{k}(x,t):=\int_{0}^{1}F(x,x\tau)P_{2k}(\tau)  d\tau$
with respect to the basis $\left\{  p_m(t)
\right\}  _{m=0}^{\infty}$.
\end{proof}

\begin{remark}
Note that the expressions $-\frac{2\omega}{\pi^{2}n}\bigl(x\sin nx\,\cos
nt+t\sin nt\,\cos nx\bigr)$, subtracted termwise from the series \eqref{F} in
order to obtain the representation \eqref{F alt}, are only the first order
approximations to the terms $\frac{\cos\rho_{n}x\,\cos\rho_{n}t}{\alpha_{n}%
}-\frac{\cos nx\,\cos nt}{\alpha_{n}^{0}}$. Hence, given a finite set of
spectral data, one can still benefit from adding some number of
\textquotedblleft asymptotic spectral data\textquotedblright\ as it was
discussed in Remark \ref{Remark add asymptotic}. However, much fewer terms are
needed to achieve a similar accuracy when comparing to the original
representation \eqref{F}.
\end{remark}

\begin{remark}
One can consider the integrated version of the representation \eqref{F alt}
and substitute it into the integrated Gelfand-Levitan equation \eqref{G-L-i}.
This results in a system similar to \eqref{G-L-alt}, with more complicated
expressions for the coefficients. Although one could expect a faster
convergence of the series involved, our numerical experiments showed that in
fact this does not lead to a significantly better accuracy.
\end{remark}

\subsection{Truncated main system of equations, its condition number,
existence and stability of solution}

\label{Subsect Stability} For the numerical solution of the system
\eqref{G-L-alt} it is natural to truncate the infinite system, i.e., to
consider $m,k\leq N$. We are going to prove that the solution of this
truncated system converges to the exact one when $N\rightarrow\infty$, and
that the process is stable, i.e., the condition number of the coefficient
matrix is uniformly bounded with respect to $N$, and, as a result, small
errors in coefficients and in the right hand side of the system lead to small
errors in the approximate solution. In order to apply the general theory, let
us multiply each equation in \eqref{G-L-alt} by $\sqrt{4k+1}\sqrt{x}$ and
rewrite the system as
\begin{equation}
\frac{g_{k}(x)}{\sqrt{4k+1}\sqrt{x}}+\sum_{m=0}^{\infty}\frac{g_{m}(x)}%
{\sqrt{4m+1}\sqrt{x}}\cdot x\sqrt{4k+1}\sqrt{4m+1}\widetilde{C}_{k,m}%
(x)=\sqrt{4k+1}\sqrt{x}\widetilde{d}_{k}(x),\quad k=0,1,\ldots\label{G-L-alt2}%
\end{equation}
Denote
\[
\xi_{k}=\frac{g_{k}(x)}{\sqrt{4k+1}\sqrt{x}},\qquad b_{k}=\sqrt{4k+1}\sqrt
{x}\widetilde{d}_{k}(x),\qquad a_{km}=x\sqrt{4k+1}\sqrt{4m+1}\widetilde
{C}_{k,m}(x).
\]
Then \eqref{G-L-alt2} can be written as
\begin{equation}
\xi_{k}+\sum_{k=0}^{\infty}a_{km}\xi_{m}=b_{k},\qquad k=0,1,\ldots
\label{G-L-alt L2form}%
\end{equation}

As was mentioned in the proof of Theorem \ref{Thm main system alt}, one has
\begin{align}
\xi_{k} &  =\int_{0}^{x}G(x,t)p_{k}(t)\,dt,\label{xik}\\
b_{k} &  =-\sqrt{4k+1}\sqrt{x}\int_{0}^{x}F(x,t)P_{2k}\left(  \frac{t}%
{x}\right)  \,\frac{dt}{x}=-\int_{0}^{x}F(x,t)p_{k}(t)\,dt,\label{bk}\\
a_{km} &  =\sqrt{x}\sqrt{4k+1}\int_{0}^{x}\int_{0}^{x}F(s,t)P_{2k}\left(
\frac{t}{x}\right)  \frac{dt}{x}\cdot p_{m}(s)\,ds=\int_{0}^{x}\int_{0}%
^{x}F(s,t)p_{k}(t)p_{m}(s)\,dt\,ds,\label{akm}%
\end{align}
where $\{p_{m}(t)\}_{m=0}^{\infty}$ is the orthonormal basis of $L_{2}(0,x)$
given by \eqref{basis pm}. Hence $\{\xi_{k}\}_{k=0}^{\infty}\in\ell_{2}$,
$\{b_{k}\}_{k=0}^{\infty}\in\ell_{2}$, $\{a_{km}\}_{k,m=0}^{\infty}\in\ell
_{2}\otimes\ell_{2}$, and the system \eqref{G-L-alt L2form} is of the type
studied in \cite[Chapter 14, \S 3]{KantorovichAkilov}. Moreover, it follows
from \eqref{xik}--\eqref{akm} that the truncated system
\begin{equation}
\xi_{k}+\sum_{k=0}^{N}a_{km}\xi_{m}=b_{k},\qquad k=0,1,\ldots
,N,\label{G-L-alt L2form truncated}%
\end{equation}
coincides with the system obtained by applying the Bubnov-Galerkin process to
equation \eqref{Gelfand-Levitan} with respect to the system \eqref{basis pm},
see \cite[\S 14]{Mihlin}. However, in our approach we do not need to solve the
complete system, only the first solution component $\xi_{0}$ is necessary to
recover the potential. Below, in Section \ref{Sect Numerical examples}\ we
illustrate the importance of this fact by a numerical example. Also we point
out that the special form of the function $F$ and the function system
\eqref{basis pm} allowed us to avoid the necessity to compute numerically the
scalar products arising in the Bubnov-Galerkin process and to reduce them to
the sums \eqref{c_km tilde} and \eqref{d_k til}.

Let $I_{N}$ be an $(N+1)\times(N+1)$ identity matrix, $L_{N}=(a_{km}%
)_{k,m=0}^{N}$ the coefficient matrix of the truncated system and
$R_{N}=(b_{k})_{k=0}^{N}$ the truncated right-hand side. Let
\begin{equation}
U_{N}=\left\{  \xi_{k}^{N}\right\}  _{k=0}^{N}=\left\{  \frac{g_{k}^{N}%
(x)}{\sqrt{4k+1}\sqrt{x}}\right\}  _{k=0}^{N} \label{sol truncated}%
\end{equation}
denote the solution of the truncated system \eqref{G-L-alt L2form truncated}.
Following \cite[\S 9]{Mihlin} consider a system (called non-exact system)
\[
(I_{N}+L_{N}+\Gamma_{N})v=R_{N}+\Delta_{N},
\]
where $\Gamma_{N}$ is an $(N+1)\times(N+1)$ matrix representing errors in the
coefficients $a_{km}$, and $\Delta_{N}$ is a column-vector representing errors
in the coefficients $b_{k}$. Let $V_{N}$ denote the solution of the non-exact
system. The solution of the Bubnov-Galerkin process is called stable if there
exist constants $c_{1}$, $c_{2}$ and $r$ independent of $N$ such that for
$\Vert\Gamma_{N}\Vert\leq r$ and arbitrary $\Delta_{N}$ the non-exact system
is solvable and the following inequality holds
\[
\Vert U_{N}-V_{N}\Vert\leq c_{1}\Vert\Gamma_{N}\Vert+c_{2}\Vert\Delta_{N}%
\Vert.
\]
Application of the general theory from \cite[Chapter 14, \S 3]%
{KantorovichAkilov} and \cite[Theorems 14.1 and 14.2]{Mihlin} leads to the
following result.

\begin{proposition}
Let $x>0$ be fixed. Then for a sufficiently large $N$ the truncated system
\eqref{G-L-alt L2form truncated} has a unique solution denoted in
\eqref{sol truncated} by $U_{N}$ and
\[
\sum_{k=0}^{N}\frac{|g_{k}^{N}(x)-g_{k}(x)|^{2}}{(4k+1)}+\sum_{k=N+1}^{\infty
}\frac{|g_{k}(x)|^{2}}{(4k+1)}\rightarrow0,\qquad N\rightarrow\infty,
\]
from which it also follows that
\[
g_{0}^{N}(x)\rightarrow g_{0}(x),\qquad N\rightarrow\infty.
\]
There exists such constant $r>0$ that for arbitrary error matrices $\Gamma_n$ satisfying $\|\Gamma_n\|<r$ for all $n\in\mathbb{N}$ the condition numbers of the approximate coefficient matrices $I_{N}+L_{N}+\Gamma_N$ are uniformly
bounded with respect to $N$ and the solution $U_{N}$ is stable.
\end{proposition}

\subsection{Flipping the interval}

\label{Subsect reverse interval} A motivation for considering  the integrated
Gelfand-Levitan equation \eqref{Gelfand-Levitan integrated} and the
alternative representation \eqref{F alt} is to solve the problem of the jump
discontinuity of the series \eqref{F} at $x=t=\pi$ (for $\omega\neq0$) and the
resulting slow convergence of the series near this critical point. An
additional improvement can be achieved by recovering the potential only on the
first half of the interval, while on the second half recovering the flipped
potential $q(\pi-x)$.

Another motivation to consider the flipped problem consists in the following
observation. Suppose the potential $q$ is piecewise smooth, that is, there are
several \textquotedblleft special\textquotedblright\ points inside $[0,\pi]$,
where some derivative of $q$ is discontinuous, while on the rest of the
segment the potential is infinitely differentiable. It is known \cite{KNT}
that the series \eqref{G Fourier series} converges fast (super polynomially)
up to the first \textquotedblleft special point\textquotedblright, and
afterwards the convergence slows down to a polynomial one. For the inverse
problem this phenomenon results in a higher accuracy of the recovered
potential on the segment from $0$ to the first \textquotedblleft special
point\textquotedblright. By performing the flipping of the problem one can
assure a higher accuracy of the recovered potential also on the segment from
the last \textquotedblleft special point\textquotedblright\ to $\pi$. The
greatest advantage is achieved in the case when there is only one
\textquotedblleft special point\textquotedblright.

In this subsection we show how the spectral data for the flipped potential can
be obtained.

Consider the following \textquotedblleft flipped\textquotedblright\ spectral
problem
\begin{gather}
-y^{\prime\prime}+q(\pi-x)y=\rho^{2}y,\qquad x\in(0,\pi),\label{Eq reversed}\\
y^{\prime}(0)-Hy(0)=0,\qquad y^{\prime}(\pi)+hy(\pi)=0. \label{BC reversed}%
\end{gather}

Obviously it has the same eigenvalues as the problem \eqref{SL equation},
\eqref{bc1}, and the functions $\varphi^{r}(\rho_{n},x):=\varphi(\rho_{n}%
,\pi-x)$ are the eigenfunctions of \eqref{Eq reversed}, \eqref{BC reversed}.
But the functions $\varphi^{r}(\rho_{n},x)$ do not necessarily satisfy
\eqref{init cond}, so the norming constants (let us denote them by $\alpha
_{n}^{r}$) change. Actually, $\varphi^{r}(\rho_{n},0)=\varphi(\rho_{n},\pi)$,
so it can be seen from \eqref{init cond} and \eqref{alpha n} that
\begin{equation}
\alpha_{n}^{r}=\frac{\alpha_{n}}{\varphi^{2}(\rho_{n},\pi)}. \label{alpha n r}%
\end{equation}

The formula \cite[formula (1.1.35)]{Yurko2007} states that
\begin{equation}
\alpha_{n}=-\frac{\varphi\left(  \rho_{n},\pi\right)  }{2\rho_{n}}\left.
\frac{d}{d\rho}\left(  \varphi^{\prime}\left(  \rho,\pi\right)  +H\varphi
\left(  \rho,\pi\right)  \right)  \right\vert _{\rho=\rho_{n}}.
\label{alpha_n}%
\end{equation}
Notice that with the aid of Theorem \ref{Theorem NSBF} equality
\eqref{alpha_n} can be written in the form
\begin{equation}
\label{alpha n partial}\alpha_{n} = \frac{\varphi(\rho_{n},\pi)}{2\rho_{n}}
\biggl( ( 1+\pi\omega) \sin\rho_{n}\pi+\pi\rho_{n}\cos\rho_{n}\pi+\sum
_{k=0}^{\infty}(-1)^{k}h_{k}\left(  \pi j_{2k+1}(\rho_{n}\pi) -\frac{2k}%
{\rho_{n}}j_{2k}(\rho_{n}\pi)\right)  \biggr)
\end{equation}
for $n=1,2,\ldots,$ and
\begin{equation}
\label{alpha 0 partial}\alpha_{0} = \varphi(0,\pi) \left(  \pi+\omega\frac
{\pi^{2}}{3}+h_{1}\frac{\pi^{2}}{15}\right)  ,
\end{equation}
here $h_{k}$ are the constants introduced in Subsection
\ref{subsect recovery hn}.

Hence
\begin{equation}
\label{alpha n reversed}\alpha_{n}^{r} = \frac{\left(  ( 1+\pi\omega) \sin
\rho_{n}\pi+\pi\rho_{n}\cos\rho_{n}\pi+\sum_{k=0}^{\infty}(-1)^{k}h_{k}\left(
\pi j_{2k+1}(\rho_{n}\pi) -\frac{2k}{\rho_{n}}j_{2k}(\rho_{n}\pi)\right)
\right)  ^{2}}{4\alpha_{n} \rho_{n}^{2}}%
\end{equation}
for $n=1,2,\ldots$ and
\begin{equation}
\label{alpha 0 reversed}\alpha_{0}^{r} = \frac{\left(  \pi+\omega\frac{\pi
^{2}}{3}+h_{1}\frac{\pi^{2}}{15}\right)  ^{2}}{\alpha_{0}}.
\end{equation}

\subsection{Algorithm of solution of Problem \ref{Problem 1}%
\label{Subsect Algorithm 1}}

Given a finite set of spectral data $\left\{  \rho_{n},\,\alpha_{n}\right\}
_{n=0}^{N_{1}}$, the following direct method for recovering the potential $q$
and the numbers $h$ and $H$ is proposed.

\begin{enumerate}
\item\label{Step0} If $\rho_{0}\neq0$, perform the shift of the eigenvalues
\[
\tilde{\rho}_{n}=\sqrt{\rho_{n}^{2}-\rho_{0}^{2}},\qquad n\geq0,
\]
so that $0$ becomes the first eigenvalue of the spectral problem. Let us
denote the shifted eigenvalues by the same expression $\rho_{n}$.

\item Find the coefficients $\omega^{(2j-1)}$, $1\le j\le K$ and
$\omega^{(2j,+)}$, $1\le j\le K-1$ as described in Subsection
\ref{Subsect recovery omega}. The parameter $\omega$ can be recovered as
$\omega= \pi\omega^{(1)}$.

\item \label{StepAsymptotic} Complement the set of ``exact'' spectral data
$\left\{  \rho_{n},\,\alpha_{n}\right\}  _{n=0}^{N_{1}}$ with a set of
``asymptotic'' spectral data $\left\{  \rho_{n}=n+\sum_{j=1}^{K}\frac
{\omega^{(2j-1)}}{ n^{2j-1}},\,\alpha_n=\frac{\pi}{2}+\sum_{j=1}^{K-1}\frac
{\omega^{(2j,+)}}{ n^{2j}}\right\}  _{n=N_{1}+1}^{M}$.

\item \label{Step2} For a set of points $\left\{  x_{l}\right\}  $ from
$(0,a]$, where $a\in(\pi/2,\pi]$, compute the approximate values of the
coefficients $\widetilde C_{km}(x)$ and $\widetilde d_{k}(x)$ for
$k,m=0,\ldots,N$ with the aid of the formulas
\eqref{c_km tilde}--\eqref{d_k til} and solve the system
\eqref{G-L-alt L2form truncated} obtaining thus $g_{0}(x)$.

\item \label{Step4} Compute $q$ on the segment $[0,a]$ from (\ref{q from g0}).
Take into account that $\varphi(0,x)$ is an eigenfunction associated with the
first eigenvalue $\lambda_{0}$ and hence does not have zeros on $[0,\pi]$
(see, e. g., \cite[Theorem 8.4.5]{Atkinson}). This justifies the division over
$\varphi(0,x)$.

\item Compute $h$ from $h=g_{0}^{\prime}(0)$.

\item\label{Step hn}
 Solving the (overdetermined) system of linear algebraic equations
\eqref{syst1 overdet} (with all $M+1$ exact and ``asymptotic'' eigenvalues),
compute the constants $h_{0},\ldots, h_{N_{1}^{o}}$.

\item Perform the flipping of the interval as explained in Subsection
\ref{Subsect reverse interval} and repeating steps \ref{Step2}--\ref{Step4}
compute $q_{rev}$ on the segment $[\pi/2,\pi]$.

\item \label{Step6} Combine $q$ and $q_{rev}$ to recover the potential on the
whole segment $[0,\pi]$.

\item Compute $H$ using (\ref{omega}). For this compute the mean of the
potential $\int_{0}^{\pi}q(t)\,dt$, and thus,%
\[
H=\omega-h-\frac{1}{2}\int_{0}^{\pi}q(t)\,dt.
\]

\item Recall that one has to add the original eigenvalue $\rho_{0}^{2}$ back
to the recovered potential to return to the original potential $q$.
\end{enumerate}

Below, in Section \ref{Sect Numerical examples} we illustrate the performance
of this algorithm with several numerical examples.

\begin{remark}
The steps \ref{Step hn}--\ref{Step6} are optional. As we show in Section \ref{Sect Numerical examples}, for many applications a potential $q$ recovered performing only the steps \ref{Step0}--\ref{Step4} can be sufficient.
\end{remark}

\section{Reduction of Problem \ref{Problem 2} to Problem \ref{Problem 1}}

Now we suppose that the eigenvales $\left\{  \lambda_{n}\right\}
_{n=0}^{\infty}$ and $\left\{  \nu_{n}\right\}  _{n=0}^{\infty}$ of the
Sturm-Liouville problems (\ref{SL equation}), (\ref{bc1}) and
(\ref{SL equation}), (\ref{bc2}), respectively, are given. We remind that
$\lambda_{0}=0$. Then on the first step we can repeat the procedure described
in Subsections \ref{Subsect recovery omega} and \ref{subsect recovery hn} and
obtain the parameter $\omega$ together with the constants $h_{n}$.

\subsection{Recovery of the parameter $\omega_{1}$ and coefficients of the
asymptotic expansion of eigenvalues $\mu_{n}$}

\label{Subsect recovery omega1} The parameter $\omega_{1}$ can be recovered
using the asymptotics of the eigenvalues. Indeed, it immediately follows from
\eqref{asympt 2} that
\[
\omega_{1} = \pi\lim_{n\to\infty} n\left(  \mu_{n}-n-\frac12\right)  .
\]

Suppose that only a finite set of eigenvalues $\mu_{0},\ldots,\mu_{N_{2}}$ is
given. Then, similarly to Subsection \ref{Subsect recovery omega}, one can
recover the parameter $\omega_{1}$ using the asymptotics of the eigenvalues
$\mu_{k}$. One has for $q\in W_{2}^{N}$, $N\geq1$,
\begin{equation}
\mu_{n}=n+\frac{1}{2}+\sum_{j=1}^{N+1}\frac{\omega_{1}^{(j)}}{n^{j}}%
+\frac{\eta_{n}}{n^{N+1}},\label{mun refined}%
\end{equation}
where $\omega_{1}^{(1)}=\frac{\omega_{1}}{\pi}$ and $\{\eta_{n}\}_{n=0}%
^{\infty}\in\ell_{2}$. Contrary to \eqref{rn refined}, the coefficients
$\omega_{1}^{(2k)}$ need not be zeros. Hence one can recover the parameter
$\omega_{1}$ by finding coefficients $\omega_{1}^{(j)}$, $1\leq j\leq K$ in
the best fit problem
\begin{equation}
\sum_{j=1}^{K}\frac{\omega_{1}^{(j)}}{n^{j}}=\mu_{n}-n-\frac{1}{2},\qquad
N_{s}\leq n\leq N_{2}.\label{best fit for omega1}%
\end{equation}
Similarly, the parameter $K<N_{2}-N_{s}$ is taken to be such that the least
squares error of the fit for $K$ is significantly better than the least
squares error for $K-1$. In practice, depending on the potential, optimal
value results to be up to 5.

From (\ref{omega1}) we compute
\begin{equation}
H=\omega-\omega_{1}. \label{H}%
\end{equation}

\subsection{Recovery of the norming constants $\alpha_{n}$}

Our goal is to calculate the norming constants with the aid of the formula
\eqref{alpha_n} and thus to reduce the problem to Problem \ref{Problem 1}.

Notice that with the aid of Theorem \ref{Theorem NSBF} equality (\ref{alpha_n}%
) can be written in the form
\begin{equation}
\label{alpha_n via series}%
\begin{split}
\alpha_{n}  &  =\frac{1}{2\rho_{n}}\left(  \cos\rho_{n}\pi+\sum_{k=0}^{\infty
}(-1)^{k}g_{k}(\pi)j_{2k}(\rho_{n}\pi)\right) \\
&  \quad\times\left(  \left(  1+\pi\omega\right)  \sin\rho_{n}\pi+\pi\rho
_{n}\cos\rho_{n}\pi+\sum_{k=0}^{\infty}(-1)^{k}h_{k} \left(  \pi j_{2k+1}%
(\rho_{n}\pi)- \frac{2k} {\rho_{n}}j_{2k}(\rho_{n}\pi)\right)  \right)
\end{split}
\end{equation}
for $n=1,2,\ldots$ and
\begin{equation}
\alpha_{0}=\left(  1+g_{0}(\pi)\right)  \left(  \pi+\omega\frac{\pi^{2}}%
{3}+h_{1}\frac{\pi^{2}}{15}\right)  . \label{alpha_0}%
\end{equation}

Hence to calculate the norming constants we only need to calculate the
coefficients $g_{k}(\pi)$.

Since $\nu_{n}=\mu_{n}^{2}$ are eigenvalues of the problem (\ref{SL equation}%
), (\ref{bc2}) we have that
\[
\varphi\left(  \mu_{k},\pi\right)  =\cos\mu_{k}\pi+\sum_{n=0}^{\infty}%
(-1)^{n}g_{n}(\pi)j_{2n}(\mu_{k}\pi)=0
\]
for all $k=0,1,\ldots$. Thus, for the coefficients $g_{n}(\pi)$ we have the
system of equations
\begin{equation}
\sum_{n=0}^{\infty}(-1)^{n}g_{n}(\pi)j_{2n}(\mu_{k}\pi)=-\cos\mu_{k}\pi,\quad
k=0,1,\ldots. \label{second system}%
\end{equation}
Solving it, we find the coefficients $g_{n}(\pi)$ and hence can calculate the
constants $\alpha_{n}$ from (\ref{alpha_n via series}) and (\ref{alpha_0}).

Similarly to Subsection \ref{subsect recovery hn}, having finite number of
eigenvalues $\{\mu_{n}\}_{n=0}^{N_{2}}$ given, it is not necessary to look for
the same number of the coefficients $g_{n}(\pi)$. One can consider an
overdetermined system
\begin{equation}
\label{syst2 overdet}\sum_{n=0}^{N_{2}^{o}}(-1)^{n}g_{n}(\pi)j_{2n}(\mu_{k}%
\pi)=-\cos\mu_{k}\pi,\qquad k=0,1,\ldots,N_{2},
\end{equation}
taking as $N_{2}^{o}$ a value for which the coefficient matrix has relatively
small condition number.

Thus, we have the constants $\left\{  \lambda_{n}\right\}  _{n=0}^{\infty}$,
$\left\{  \alpha_{n}\right\}  _{n=0}^{\infty}$ together with the parameter
$\omega$ and the constant $H$, and can apply the algorithm from Subsection
\ref{Subsect Algorithm 1} for recovering $q(x)$ and $h$.

\begin{remark}
\label{Rmk alternate omega1} We obtain from \eqref{Gxx} and \eqref{Gxx series}
that
\begin{equation}
\omega_{1}=G(\pi,\pi)=\sum_{n=0}^{\infty}\frac{g_{n}(\pi)}{\pi},
\label{omega1 via series}%
\end{equation}
an alternative way to recover the parameter $\omega_{1}$.
\end{remark}

\subsection{Algorithm of solution of Problem \ref{Problem 2}}

\label{Subsect Algorithm 2}

Given two finite sets of eigenvalues $\left\{  \rho_{n}^{2}\right\}
_{n=0}^{N_{1}}$ and $\left\{  \mu_{n}^{2}\right\}  _{n=0}^{N_{2}}$ of the
Sturm-Liouville problems (\ref{SL equation}), (\ref{bc1}) and
(\ref{SL equation}), (\ref{bc2}), respectively. $N_{2}$ can differ from
$N_{1}$. The following algorithm allows one to approximate $q(x)$ and the
constants $h$, $H$.

\begin{enumerate}
\item Find the coefficients $\omega^{(2j-1)}$, $1\le j\le K_{1}$ and
$\omega_{1}^{(j)}$, $1\le j\le K_{2}$ as described in Subsections
\ref{Subsect recovery omega} and \ref{Subsect recovery omega1}. The parameters
$\omega$ and $\omega_{1}$ can be recovered as $\omega= \pi\omega^{(1)}$,
$\omega_{1} = \pi\omega_{1}^{(1)}$.

\item Compute $H$ from (\ref{H}).

\item \label{StepAsymptotic2} Complement the set of ``exact'' spectral data
$\left\{  \rho_{n}\right\}  _{n=0}^{N_{1}}$, $\left\{  \mu_{n}\right\}
_{n=0}^{N_{2}}$ with a set of ``asymptotic'' spectral data $\left\{  \rho
_{n}=n+\sum_{j=1}^{K_{1}}\frac{\omega^{(2j-1)}}{ n^{2j-1}}\right\}
_{n=N_{1}+1}^{M}$, $\left\{  \mu_{n}=n+\frac12+\sum_{j=1}^{K_{2}}\frac
{\omega_{1}^{(j)}}{ n^{j}}\right\}  _{n=N_{2}+1}^{M}$.

\item Solving the (overdetermined) systems \eqref{syst1 overdet} and
\eqref{syst2 overdet} (with both given and ``asymptotic'' eigenvalues) compute
the constants $h_{n}$ for $n=0,\ldots, N_{1}^{o}$ and $g_{n}(\pi)$ for
$n=0,\ldots,N_{2}^{o}$.

\item Compute the norming constants $\alpha_{0}$ by (\ref{alpha_0}) and
\begin{align*}
\alpha_{n}  &  =\frac{1}{2\rho_{n}}\biggl( \cos\rho_{n}\pi+\sum_{k=0}%
^{N_{2}^{o} }(-1)^{k}g_{k}(\pi)j_{2k}(\rho_{n}\pi)\biggr)\\
&  \quad\times\biggl( \left(  1+\pi\omega\right)  \sin\rho_{n}\pi+\pi\rho
_{n}\cos\rho_{n}\pi+\sum_{k=0}^{N_{1}^{o}}(-1)^{k}h_{k}\left(  \pi
j_{2k+1}(\rho_{n}\pi)-\frac{2k}{\rho_{n}}j_{2k}(\rho_{n}\pi)\right)  \biggr)
\end{align*}
for $n=1,\ldots,M$.

\item Apply steps \ref{Step2}--\ref{Step6} from Subsection
\ref{Subsect Algorithm 1} for computing $q(x)$ and $h$.
\end{enumerate}

\subsection{Modification for other spectral data}

Though in the present work only two classical inverse Sturm-Liouville problems
are considered, the method is applicable to some other inverse Sturm-Liouville
problems. For example, it is well known (see, e.g., \cite[p. 103]{Chadan et al
1997}) that the problem of recovery of a symmetric potential from one
spectrum reduces to a two-spectra problem.

The problem of recovery of the potential from one spectrum and endpoint data
can be solved by our methos as well. Suppose that besides the spectrum of a
problem (\ref{SL equation}), (\ref{bc2}) a sequence of endpoint data is given,
say,  $\left\{  \varphi^{\prime}(\mu_{n},\pi)\right\}  _{n=0}^{\infty}$. Then
the set of constants $\left\{  g_{n}(\pi)\right\}  $ can be recovered from
(\ref{second system}), and consequently the parameter $\omega_{1}$ from
(\ref{omega1 via series}). The set of constants $\left\{  \gamma_{n}%
(\pi)\right\}  $ can be recovered from the system
\[
\sum_{n=0}^{\infty}(-1)^{n}\gamma_{n}(\pi)j_{2n}(\mu_{k}\pi)=\varphi^{\prime
}(\mu_{k},\pi)+\mu_{k}\sin\mu_{k}\pi-\omega_{1}\cos\mu_{k}\pi,\quad
k=0,1,\ldots,
\]
obtained from (\ref{phiprime}). Due to Theorem \ref{Theorem NSBF}, having the
sequences $\left\{  g_{n}(\pi)\right\}  $ and $\left\{  \gamma_{n}%
(\pi)\right\}  $ is equivalent to have $\varphi(\rho,\pi)$ and $\varphi
^{\prime}(\rho,\pi)$ for all $\rho$. This means that a second spectrum
(corresponding to an arbitrary $H$) can be computed , and hence the problem
reduces to the two-spectra problem.

\section{Numerical examples\label{Sect Numerical examples}}

Numerical implementation of the algorithms proposed in Subsections
\ref{Subsect Algorithm 1} and \ref{Subsect Algorithm 2} is quite simple and
does not require more than build-in functions of, e.g., Matlab. Only several
comments should be made.

First, we obtained the parameter $\omega$ from the asymptotics (see Subsection \ref{Subsect recovery omega}) and similarly to Subsection \ref{Subsect Stability}, we considered
$\frac{g_{n}(\pi)}{\sqrt{4n+1}}$ and $\frac{h_{n}}{\sqrt{4n+1}}$ as the new
unknowns for the systems \eqref{syst1 overdet} and \eqref{syst2 overdet},
rewriting them as
\begin{gather*}
\sum_{n=0}^{N_{1}^{o}}(-1)^{n}\frac{h_{n}}{\sqrt{4n+1}%
}\sqrt{4n+1}j_{2n}(\rho_{k}\pi)=-\omega\cos\rho_{k}\pi+\rho_{k}\sin\rho_{k}\pi,\qquad k=1,\ldots,M,\\
\sum_{n=0}^{N_{2}^{o}}(-1)^{n}\frac{g_{n}(\pi)}{\sqrt{4n+1}}\sqrt{4n+1}%
j_{2n}(\mu_{k}\pi)=-\cos\mu_{k}\pi,\qquad k=0,1,\ldots,M.
\end{gather*}
After such modification all (except the largest one which is close to $1$ and
possibly several smallest ones) singular values of the coefficient matrices
become very close to $\frac{\sqrt{2}}{2}$. Reducing the numbers $N_{1}^{o}$
and $N_{2}^{o}$ the smallest eigenvalues disappear, and all the eigenvalues
become larger than $1/2$. So one can take as a criterion for choosing the
parameters $N_{1}^{o}$ and $N_{2}^{o}$, e.g., that the corresponding condition
numbers be less than 10.

Second, in many formulas we need the values of the spherical Bessel functions
$j_{k}(t)$ for a list of indices $k=0,\ldots,N$ for the same argument $t$. A
considerable speedup is achieved by applying the backwards recursion formula
\[
j_{n-1}(t)=\frac{2n+1}{t}j_{n}(t)-j_{n+1}(t),
\]
see, e.g., \cite{Barnett}, \cite{GillmanFiebig} for details. So only
$j_{N-1}(t)$ and $j_{N}(t)$ have to be computed using, e.g., \texttt{besselj}
function from Matlab. A special care should be taken only for small values of
$t$ and large values of $N$, since the absolute values of $j_{N-1}(t)$ and
$j_{N}(t)$ can be less than $10^{-307}$, that is zero in the double machine
precision arithmetics.

Third, numerical differentiation is used to recover the potential $q$ and the
parameter $h$ from the coefficient $g_{0}$. To obtain good results, the set of
points $\{x_{l}\}$ chosen on the step \ref{Step2}, should not be large,
100--200 points are sufficient. We refer the reader to \cite[p.111--114]%
{VS1986}, where an explanation is given for a closely related problem of
solution of the Volterra integral equation of the first kind. We have
interpolated the obtained values $g_{0}(x_{l})$ by a 6th order spline and used
Matlab's function \texttt{fnder} to perform the differentiation. It should be
noted that adding a large set of \textquotedblleft
asymptotic\textquotedblright\ spectral data (on steps \ref{StepAsymptotic} of the proposed algorithms) greatly
improves the stability of numerical differentiation in the sense that the
obtained potential $q$ changes very little whenever one takes 50, 100 or 500
points $\{x_{l}\}$.

Another possibility to perform numerical differentiation is to approximate the
coefficient $g_{0}$ by a polynomial on the whole segment $[0,\pi]$ and
differentiate this polynomial. By taking $\{x_{l}\}$ to be the Chebyshev nodes
(adapted to the segment $[0,\pi]$), coefficients of the Fourier-Chebyshev
series of $g_{0}$ can be obtained quite easily. Higher-order oscillatory
component in $g_{0}$ (appearing due to truncation of the series in
\eqref{c_km tilde}--\eqref{d_k til}), leading to large error during numerical
differentiation, can be filtered out by truncating the Fourier-Chebyshev
series once the coefficients become relatively small. The remaining partial
sum provides an approximating polynomial. We refer the reader to \cite{AT2017}
for further details.

\subsection{Numerical solution of Problem \ref{Problem 1}}

In this subsection we illustrate the algorithm from Subsection
\ref{Subsect Algorithm 1} on two examples: of a smooth potential and of a
potential possessing a discontinuous derivative. We have taken the same
potentials as those considered in \cite{Kammanee Bockman 2009},
\cite{Kr2019JIIP}, \cite{KrBook2020}, however having both parameters $h$ and
$H$ different from zero to illustrate that proposed method works in the most
general setting.

\begin{example}
\label{Ex1} Consider the spectral problem \eqref{SL equation}, \eqref{bc1}
with $q(x)=2 + \sin2x$, $h=1$ and $H=1/2$. We have computed the spectral data
$\{\rho_{n},\, \alpha_{n}\}_{n=0}^{N_{1}}$ for $N_{1}=200$ using the method
proposed in \cite{KNT}.

The parameter $\omega$ was recovered as explained in Subsection
\ref{Subsect recovery omega} with $K=1,2,3$. The following accuracy was
obtained: for $K=1$: $5.2\cdot10^{-5}$, for $K=2$: $5.7\cdot10^{-9}$ and for
$K=3$: $7.7\cdot10^{-12}$. There was observed no further improvement when
taking $K>3$.

As was mentioned in Remark \ref{Rmk Alternative omega}, the parameter $\omega$ also can be recovered from $\omega = -h_0$, without need to use eigenvalues asymptotics. In Table \ref{Ex1Table1} we present absolute errors for both methods for different number of eigenvalues used. As one can see, the accuracy of the both methods is comparable.

\begin{table}[h]
\centering
\begin{tabular}{c|c|c|c}
\hline
$N_1$  & $\Delta\omega$, asymptotic & $\Delta\omega$, $\operatorname{cond}(A)<2000$ & $\Delta\omega$, $\operatorname{cond}(A)<200$ \\
\hline
201 & $7.7\cdot 10^{-12}$ & $2.2\cdot 10^{-10}$ & $9.9\cdot 10^{-12}$ \\
100 & $3.4\cdot 10^{-11}$ & $5.1\cdot 10^{-11}$ & $2.1\cdot 10^{-11}$ \\
50 & $4.0\cdot 10^{-9}$ & $6.2\cdot 10^{-12}$ & $1.0\cdot 10^{-13}$ \\
20 & $1.0\cdot 10^{-6}$ & $2.4\cdot 10^{-10}$ & $1.4\cdot 10^{-7}$ \\
10 & $6.2\cdot 10^{-5}$ & $1.1\cdot 10^{-5}$ & $5.2\cdot 10^{-5}$ \\
5  & $7.4\cdot 10^{-3}$ & $4.6\cdot 10^{-2}$ & $4.6\cdot 10^{-2}$ \\
\hline
\end{tabular}
\caption{Absolute error of the recovered parameter $\omega$ from Example \ref{Ex1}. Second column: eigenvalue asymptotics (Subsection \ref{Subsect recovery omega}) used, third and forth columns: by solving an overdetermined system \eqref{syst1 overdet}, with the number of unknowns obtained by checking condition number of the coefficient matrix.}
\label{Ex1Table1}
\end{table}

\begin{figure}[h]
\centering
\includegraphics[bb=0 0 216 173
height=2.4in,
width=3in
]{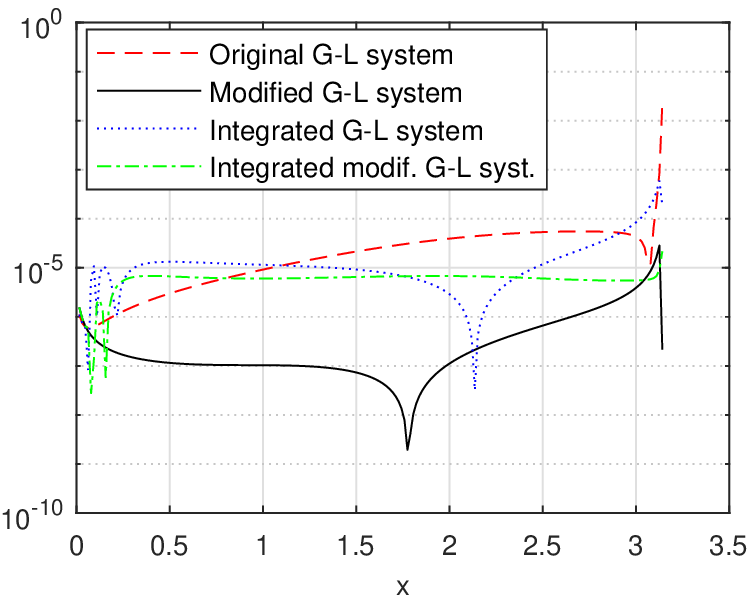} \quad\includegraphics[bb=0 0 216 173
height=2.4in,
width=3in
]{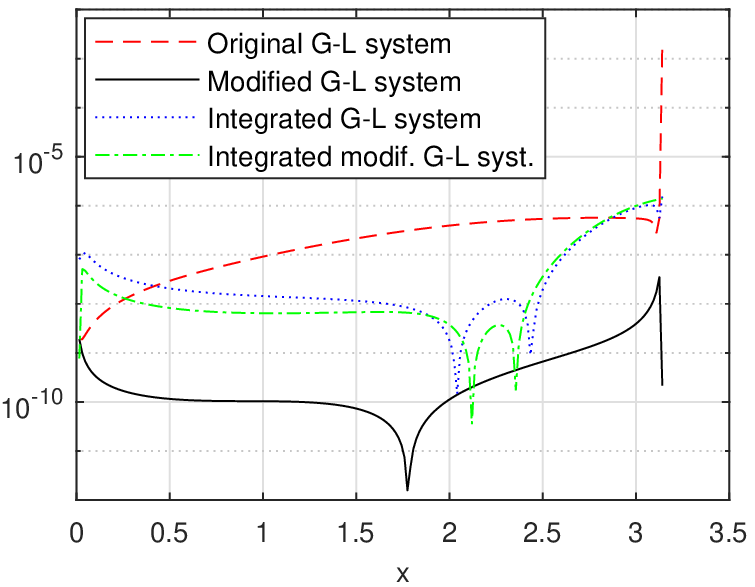}\caption{Absolute error of the recovered coefficient $g_{0}(x)$
for Example \ref{Ex1}. 4 different main systems of equations were used:
\eqref{G-L-alt} having $\omega=0$ (original G-L system), \eqref{G-L-i}
(integrated G-L system), \eqref{G-L-alt} (modified G-L system) and the one
which can be obtained if one integrates first the expression \eqref{F alt}
(integrated modified G-L system). 8 equations left in the truncated system in
all cases. On the left: 201 exact eigenvalues and norming constants were used
to compute the coefficients of the systems. On the right: additionally 1800
asymptotic eigenvalues and norming constants were added.}%
\label{Ex1Fig1}%
\end{figure}

\begin{figure}[h]
\centering
\includegraphics[bb=0 0 216 173
height=2.4in,
width=3in
]{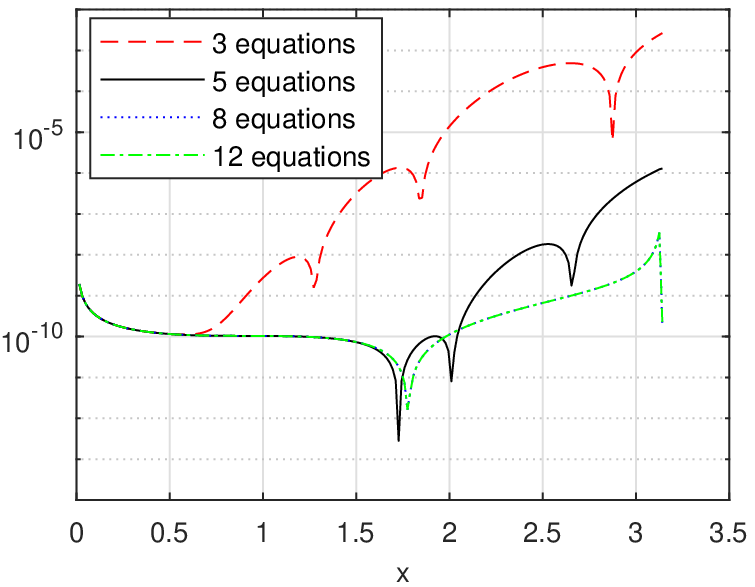} \quad\includegraphics[bb=0 0 216 173
height=2.4in,
width=3in
]{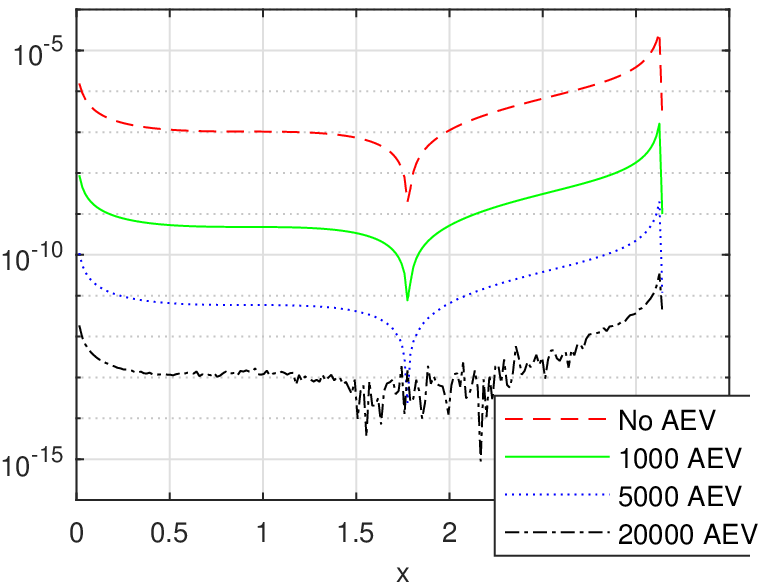}\caption{Absolute error of the recovered coefficient $g_{0}(x)$ for
Example \ref{Ex1} obtained from the truncated system \eqref{G-L-alt}. On the
left: different number of equations in the truncated system, 201 exact and
1800 asymptotic eigenvalues and norming constants were used. One the right: 8
equations in the truncated system, 201 exact but different number of
asymptotic eigenvalues (AEV) and norming constants were used.}%
\label{Ex1Fig2}%
\end{figure}

\begin{figure}[h]
\centering
\includegraphics[bb=0 0 216 173
height=2.4in,
width=3in
]{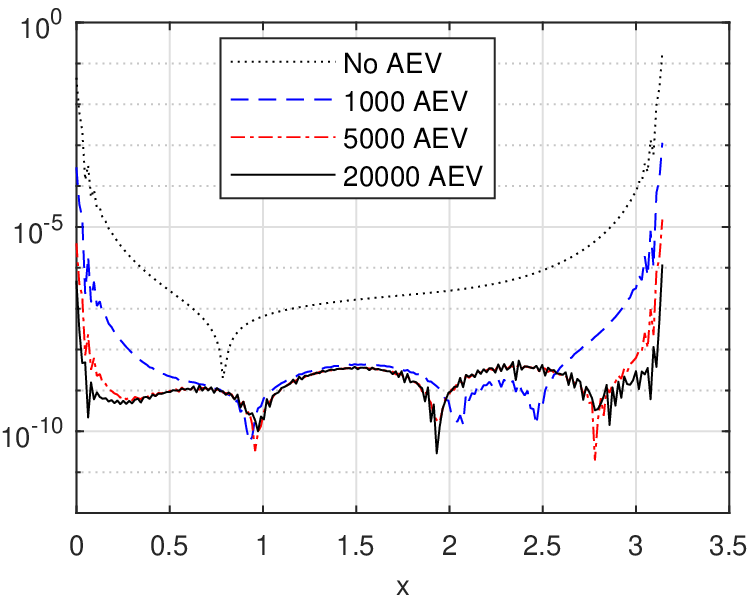} \quad\includegraphics[bb=0 0 216 173
height=2.4in,
width=3in
]{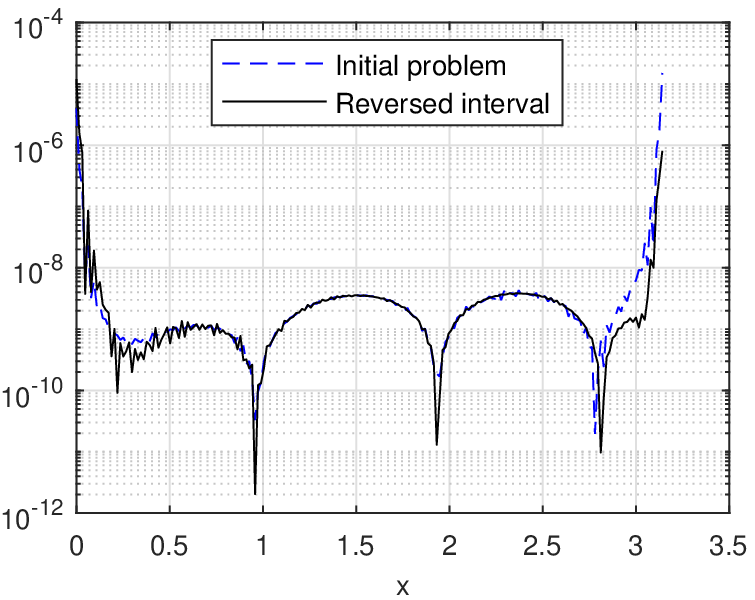}\caption{Absolute error of the recovered potential for Example
\ref{Ex1} obtained from the truncated system \eqref{G-L-alt} with 8 equations.
On the left: different number of asymptotic eigendata were added. On the
right: comparison with the error of the potential recovered from the flipped
problem (see Subsection \ref{Subsect reverse interval}), for both cases 5000
asymptotic eigendata were added.
}%
\label{Ex1Fig3}%
\end{figure}

On Figure \ref{Ex1Fig1} we compare the accuracy of the coefficient $g_{0}$
recovered using different truncated systems of equations: the original system
from \cite{Kr2019JIIP} (coinciding with \eqref{G-L-alt} when one takes
$\omega=0$), the system \eqref{G-L-i} and the system \eqref{G-L-alt}.
Additionally we considered the system which one can obtain integrating the
expression \eqref{F alt} first. As one can see, the system \eqref{G-L-alt}
provides the most accurate approximation. On Figure \ref{Ex1Fig2} we present
the absolute error of the recovered coefficient $g_{0}$ depending on the
number of equations in the truncated system and on the number of added
asymptotic spectral data. Only eight equations were sufficient to stabilize
the error, while for such smooth potentials one can observe constant accuracy
improvement even adding tens of thousands of asymptotic eigenvalues and
norming constants.

On Figure \ref{Ex1Fig3} we present the absolute error of the recovered
potential. As one can see, the effect of added asymptotic eigendata is much
smaller, and it is more evident near the interval endpoints. Flipping the
interval as was explained in Subsection \ref{Subsect reverse interval}
improves the accuracy near $x=\pi$ by one decimal digit. In all the
experiments the coefficient $g_{0}$ was computed on a uniform mesh of 201 points.

Finally, we recovered the potential $q$ and the constants $h$ and $H$ using
the truncated system \eqref{G-L-alt} and adding 20000 asymptotic eigenvalues
and norming constants. The $L_{1}$ error of the recovered potential was
$1.1\cdot10^{-8}$, absolute errors of the constants $h$ and $H$ --
$1.8\cdot10^{-9}$ and $2.3\cdot10^{-9}$, respectively.
The computation time of the proposed algorithm, even with this huge amount of added asymptotic eigenvalues, was less then 30 seconds on a laptop computer. By taking 5000 asymptotic eigendata pairs only, the $L_1$ error of the recovered potential increased to $5.0\cdot 10^{-8}$, while the computation time decreased to 7 seconds.

In many practical situations only a few eigenvalues and norming constants are
available. Similarly to examples from \cite[Section 13.3]{KrBook2020} we took
only 11 pairs $\{\rho_{n}, \alpha_{n}\}$, complemented them with 1990
asymptotic values and solved the same inverse problem. On Figure \ref{Ex1Fig4}
we present the recovered potential and its absolute error.

\begin{figure}[h]
\centering
\includegraphics[bb=0 0 216 173
height=2.4in,
width=3in
]{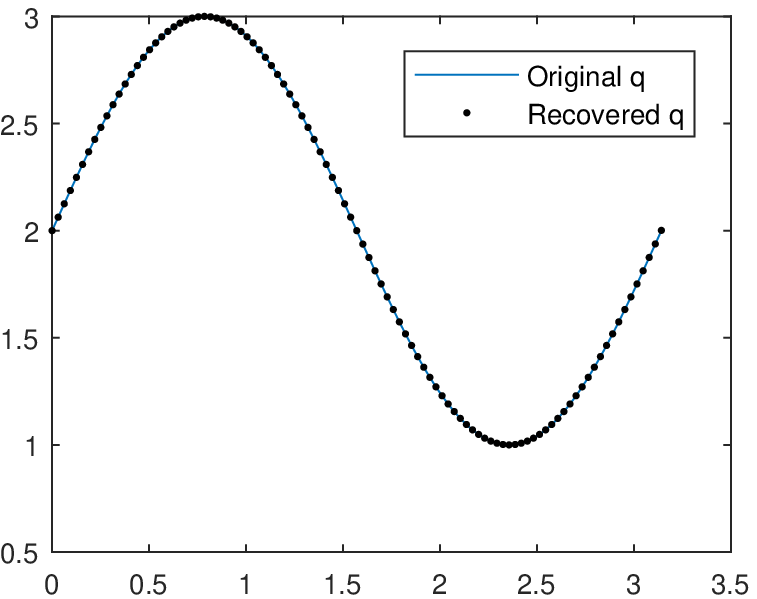} \quad\includegraphics[bb=0 0 216 173
height=2.4in,
width=3in
]{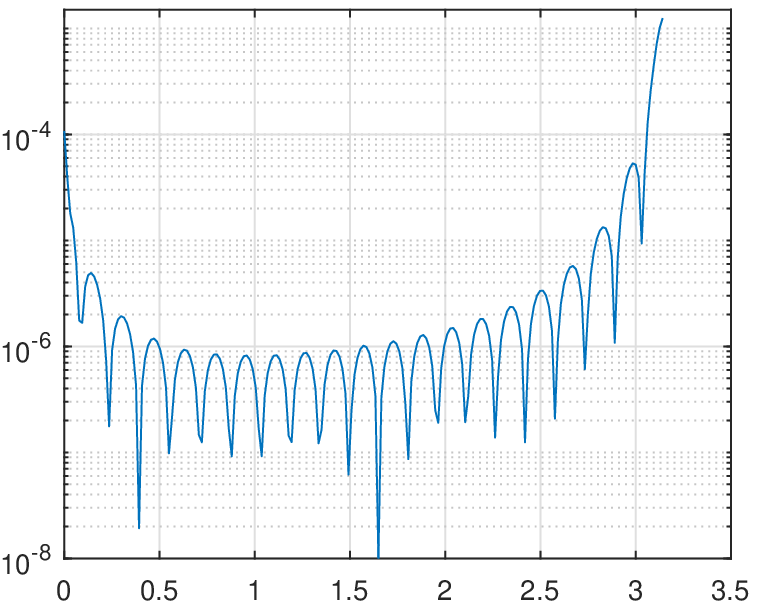}\caption{Recovered potential (on the left) and its absolute error
(on the right) for Example \ref{Ex1} obtained from the truncated system
\eqref{G-L-alt} with 8 equations. 11 exact and 1990 asymptotic eigendata pairs
were used.}%
\label{Ex1Fig4}%
\end{figure}
\end{example}

\begin{example}
\label{Ex2} Consider the spectral problem \eqref{SL equation}, \eqref{bc1}
with
\begin{equation}
\label{ex q abs}q(x)=\bigl|3-|x^{2}-3|\bigr|,
\end{equation}
$h=1$ and $H=2$. We have computed the spectral data $\{\rho_{n},\, \alpha
_{n}\}_{n=0}^{N_{1}}$ for $N_{1}=200$ using the method proposed in \cite{KNT}.

\begin{figure}[h]
\centering
\includegraphics[bb=0 0 216 173
height=2.4in,
width=3in
]{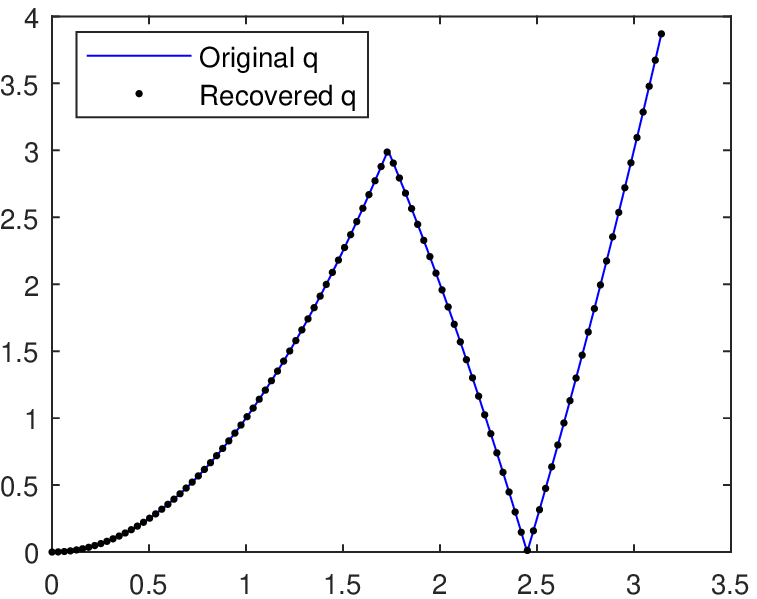} \quad\includegraphics[bb=0 0 216 173
height=2.4in,
width=3in
]{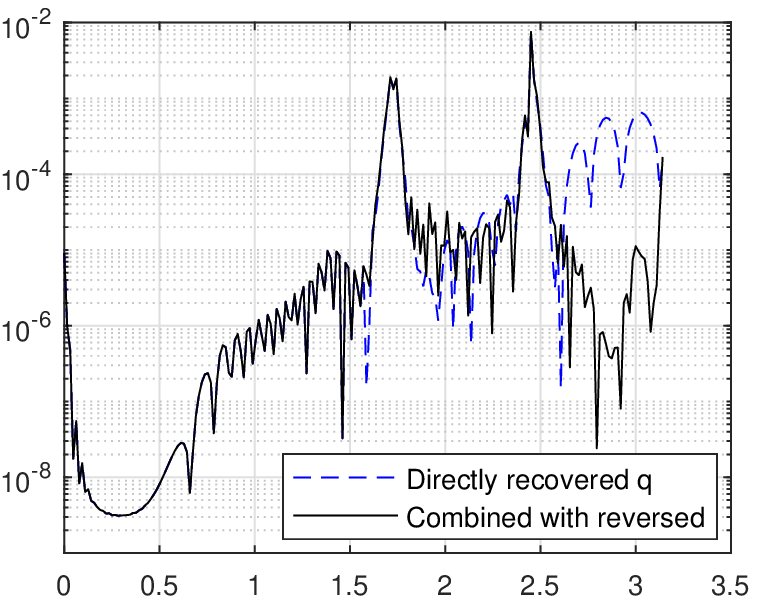}\caption{Recovered potential (on the left) and its absolute error
(on the right, dashed blue line -- obtained directly from \eqref{q from g0},
solid black line -- improved near $x=\pi$ using the flipped interval) for
Example \ref{Ex2} obtained from the truncated system \eqref{G-L-alt} with 8
equations. 201 exact and 4800 asymptotic eigendata pairs were used.}%
\label{Ex2Fig1}%
\end{figure}

On Figure \ref{Ex2Fig1} we present the recovered potential $q$ and its
absolute error, only 8 equations were used. The potential $q$ was recovered
with $L_{1}(0,\pi)$ error of $3.9\cdot10^{-4}$. The constants $\omega$, $h$
and $H$ were recovered with absolute errors of $3.7\cdot10^{-6}$,
$6\cdot10^{-8}$ and $4\cdot10^{-6}$, respectively. On Figure \ref{Ex2Fig3} we
show the result of the proposed algorithm when only 11 pairs of eigendata (to
compare with \cite{KrBook2020}) or 40 pairs of eigendata (to compare with
\cite{Kammanee Bockman 2009}) were used.

\begin{figure}[h]
\centering
\includegraphics[bb=0 0 216 173
height=2.4in,
width=3in
]{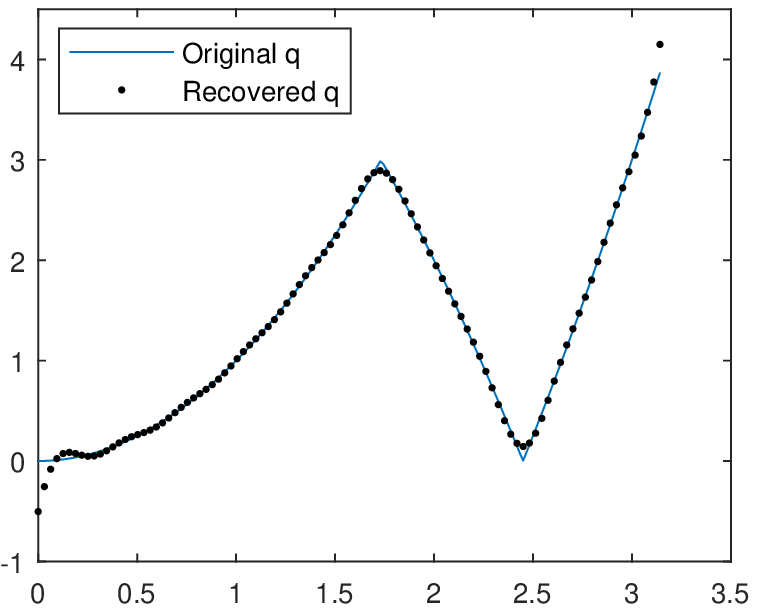} \quad\includegraphics[bb=0 0 216 173
height=2.4in,
width=3in
]{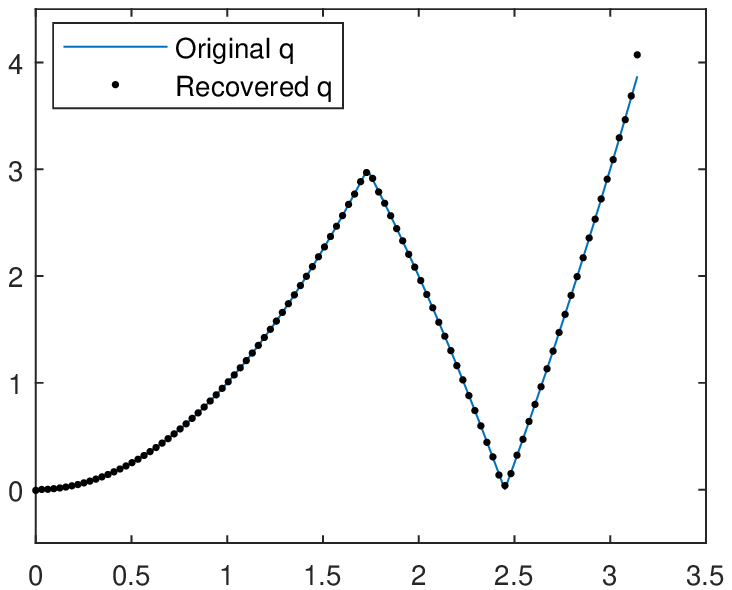}\caption{Recovered potential $q$ (on the right) for Example
\ref{Ex2} obtained from the truncated system \eqref{G-L-alt} with 8 equations.
On the left: 11 exact and 990 asymptotic eigendata pairs were used. On the
right: 40 exact and 961 asymptotic eigendata pairs were used.}%
\label{Ex2Fig3}%
\end{figure}

This example can be used to demonstrate the advantage of recovering the
coefficient $g_{0}$ only and not looking for the whole solution $G(x,t)$ of
the Gelfand-Levitan equation, recovering the potential $q$ from the expression
$G(x,x)=\sum_{n=0}^{\infty}g_{n}(x)/x$, see \eqref{Gxx series} and
\eqref{Gxx}. Indeed, the coefficients $g_{n}(x)$ for this potential decay fast
for $x<\sqrt{3}$ and slowly for $x>\sqrt{3}$ as $n\rightarrow\infty$. Here
$\sqrt{3}$ is the first point where the derivative is discontinuous. For
example, the value of $g_{500}(x)$ for $x>\sqrt{3}+0.1$ oscillates in sign
reaching values larger than $10^{-6}$. We refer the reader to \cite[(3.9) and
(4.9)]{KNT} for some estimates. For that reason, few coefficients $g_{n}(x)$
are necessary to approximate the kernel $G(x,x)$ for $x<\sqrt{3}$, but even
hundreds are not sufficient for a comparable approximation of the kernel
$G(x,x)$ for $x>\sqrt{3}$. Additionally, we do not have exact values of the
coefficients $g_{n}$, they are obtained by solving a truncated system
\eqref{G-L-alt L2form truncated}, containing some errors in the coefficients
and in the right hand side due to round-off errors, series truncation etc. So
while more coefficients $g_{n}(x)$ are necessary to be able to closely
approximate the function $G(x,x)$, summing them up results in a still larger
error due to imperfections in each computed coefficient $g_{n}(x)$. Take aside
the amount of computation time required to calculate all the coefficients and
solve the system \eqref{G-L-alt L2form truncated} with hundreds of equations.
On the contrary, the absolute error of the coefficient $g_{0}$ alone is rather
small, compensating the necessity of one more numerical differentiation, and
it can be obtained from the system of as few as 5--8 equations. We illustrate
this on Figure \ref{Ex2Fig2}. As one can see, even using 11 coefficients
$g_{n}(x)$ to approximate $G(x,x)$ for $x<\sqrt{3}$ and 251 coefficients for
$x>\sqrt{3}$ did not allow us to recover the potential $q$ with the same
accuracy as in the case when only the coefficient $g_{0}$ was used, c.f.,
Figure \ref{Ex2Fig1}.

\begin{figure}[h]
\centering
\includegraphics[bb=0 0 216 173
height=2.4in,
width=3in
]{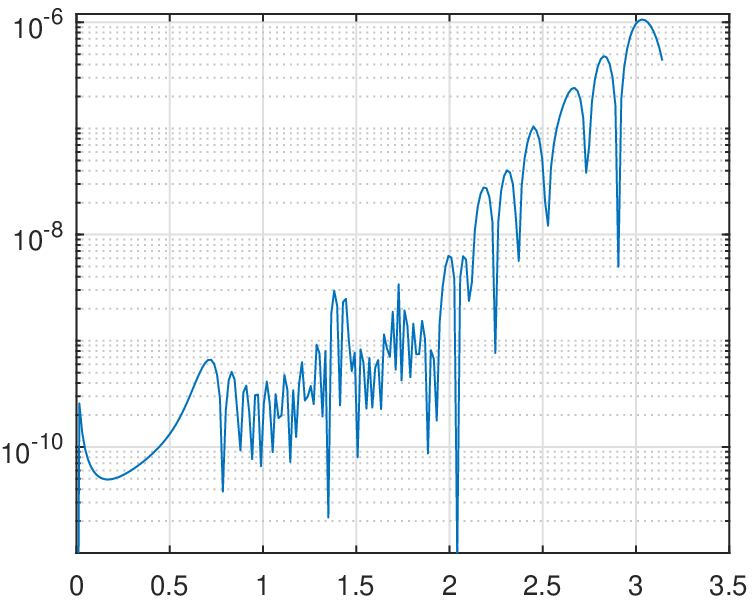} \quad\includegraphics[bb=0 0 216 173
height=2.4in,
width=3in
]{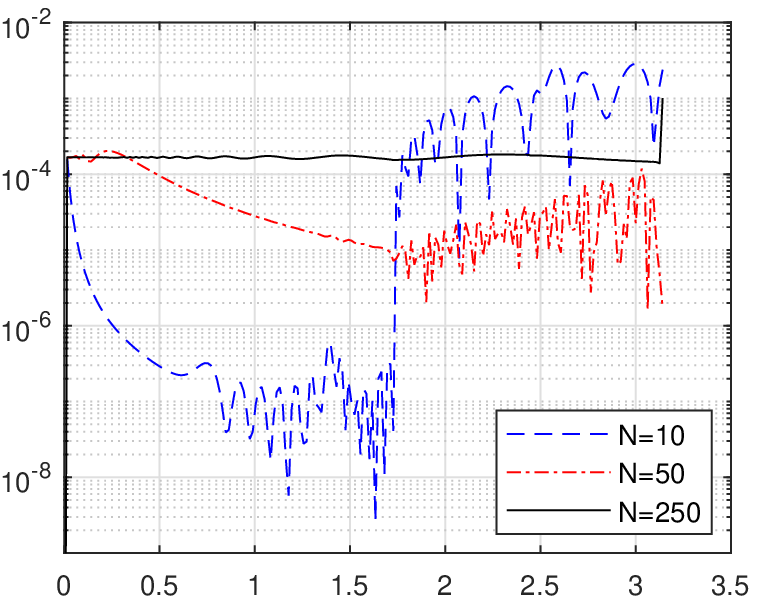}\\
\medskip
\includegraphics[bb=0 0 216 173
height=2.4in,
width=3in
]{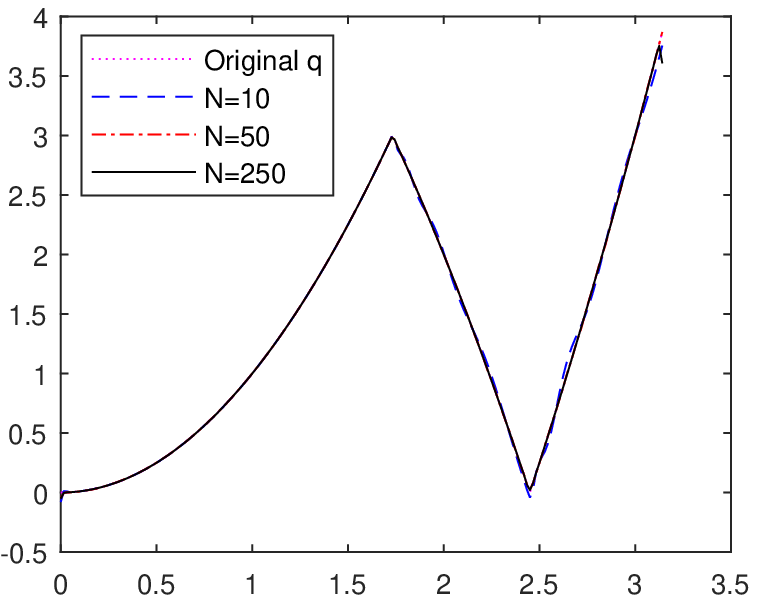} \quad\includegraphics[bb=0 0 216 173
height=2.4in,
width=3in
]{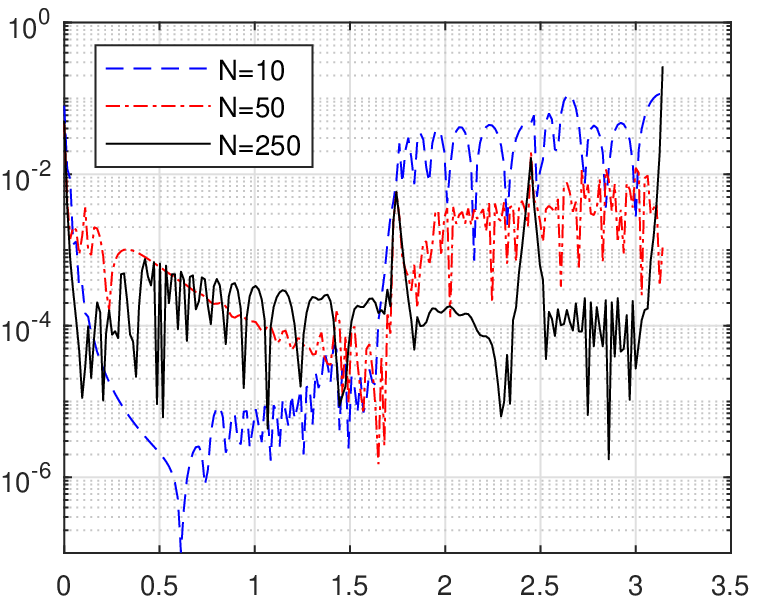}\caption{Top row: absolute error of the recovered
coefficient $g_{0}(x)$ (on the left) and the difference $\left|  G(x,x) -
\sum_{n=0}^{N} g_{n}(x)/x\right|  $ (on the right, for different values of
$N$). The coefficients $g_{n}$, $n\le N$ where obtained from the system
\eqref{G-L-alt L2form truncated} having the same truncation parameter $N$.
Bottom row: potential $q$ recovered using \eqref{Gxx}, calculated using the
coefficients $g_{n}$ with $n\le N$ (on the left) and its absolute error (on
the right). 201 exact and 4800 asymptotic eigendata pairs were used.}%
\label{Ex2Fig2}%
\end{figure}
\end{example}

\subsection{Numerical solution of Problem \ref{Problem 2}}

In this subsection we illustrate the algorithm from Subsection
\ref{Subsect Algorithm 2}. The performance of the algorithm applied for the
smooth potential from Example \ref{Ex1} was similarly excellent, being able to
recover both the potential and the constants $h$ and $H$ with accuracy better
than $10^{-8}$. For that reason we decided to present only non-smooth potentials.

\begin{example}
\label{Ex4} Consider the spectral problems \eqref{SL equation}, \eqref{bc1}
and \eqref{SL equation}, \eqref{bc2} with $h=1$, $H=2$ and potentials,
possessing discontinuous derivatives. We take the same potential
\eqref{ex q abs} and the following \textquotedblleft
sawtooth\textquotedblright\ potential
\begin{equation}
q_{2}(x)=\int_{0}^{x}\operatorname{sign}\left(  \sin\left(  \frac{10t}%
{4-t}\right)  \right)  \,dt,\label{ex q sawtooth}%
\end{equation}
inspired by an excellent article \cite{Trefethen2011}. For both potentials we
considered 40 or 201 eigenvalue pairs, complemented them with asymptotic
eigenvalues to get in total 10001 pairs and solved the system \eqref{G-L-alt}
truncated to eight equations. On Figure \ref{Ex4Fig1} we present the recovered
potentials together with $L_{1}(0,\pi)$ norm of the absolute errors and
obtained values of the parameters $h$ and $H$.

\begin{figure}[h]
\centering
\begin{tabular}
[c]{cc}%
\multicolumn{2}{c}{Potential \eqref{ex q abs}}\\
201 eigenvalue pairs. & 40 eigenvalue pairs\\
$\|q-q_{rec}\|_{L_{1}(0,\pi)} = 4.4\cdot10^{-4}$ & $\|q-q_{rec}\|_{L_{1}%
(0,\pi)} = 1.2\cdot10^{-2}$\\
$h_{rec} = 0.9999929$, $H_{rec} = 2.000055$ & $h_{rec} = 1.0011$, $H_{rec} =
1.9954$\\
\includegraphics[bb=0 0 216 173
height=2.4in,
width=3in
]{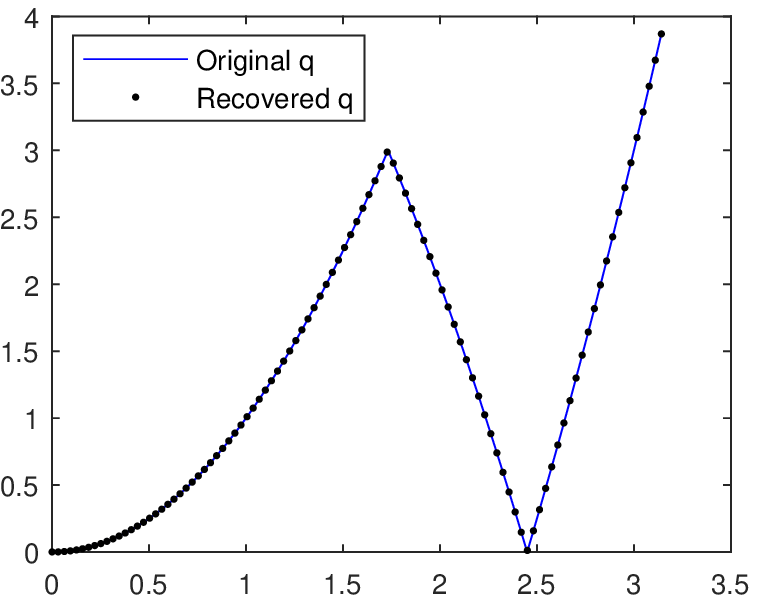} & \includegraphics[bb=0 0 216 173
height=2.4in,
width=3in
]{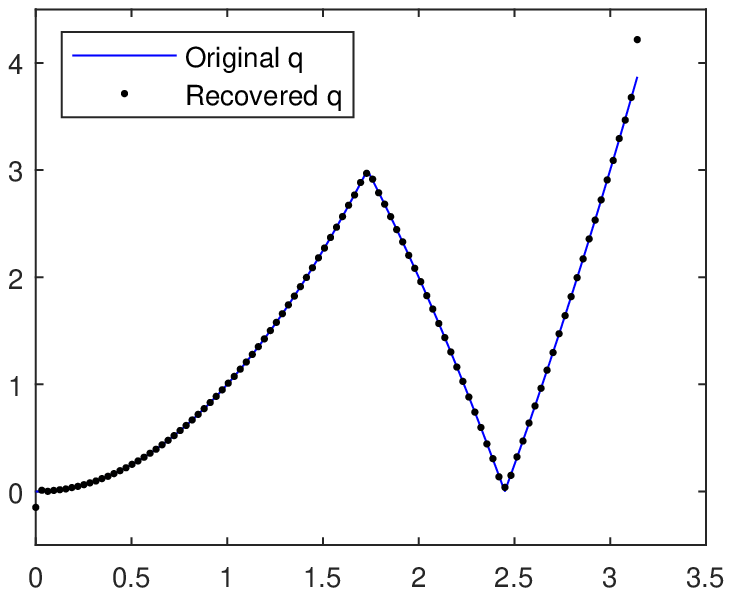}\\
\medskip\\
\multicolumn{2}{c}{Potential \eqref{ex q sawtooth}}\\
201 eigenvalue pairs. & 40 eigenvalue pairs\\
$\|q-q_{rec}\|_{L_{1}(0,\pi)} = 6.1\cdot10^{-4}$ & $\|q-q_{rec}\|_{L_{1}%
(0,\pi)} = 5.0\cdot10^{-3}$\\
$h_{rec} = 1.0000031$, $H_{rec} = 1.999980$ & $h_{rec} = 1.00031$, $H_{rec} =
1.99896$\\
\includegraphics[bb=0 0 216 173
height=2.4in,
width=3in
]{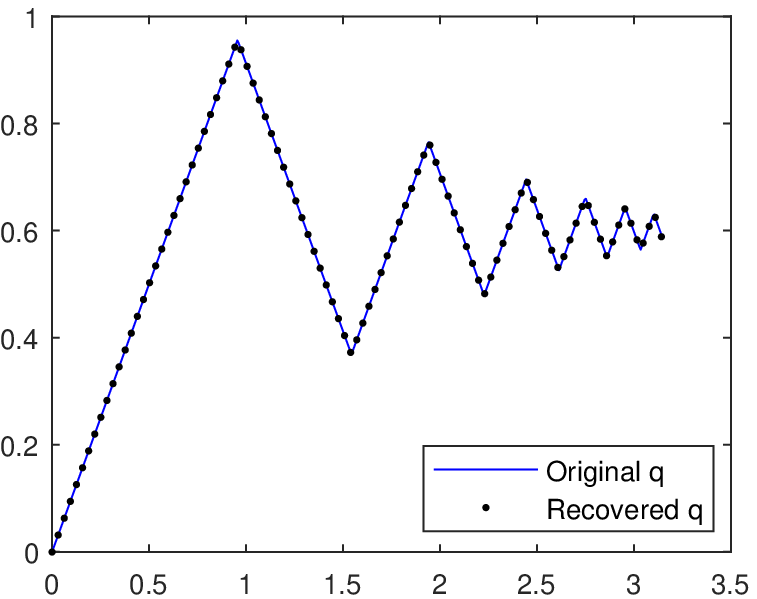} & \includegraphics[bb=0 0 216 173
height=2.4in,
width=3in
]{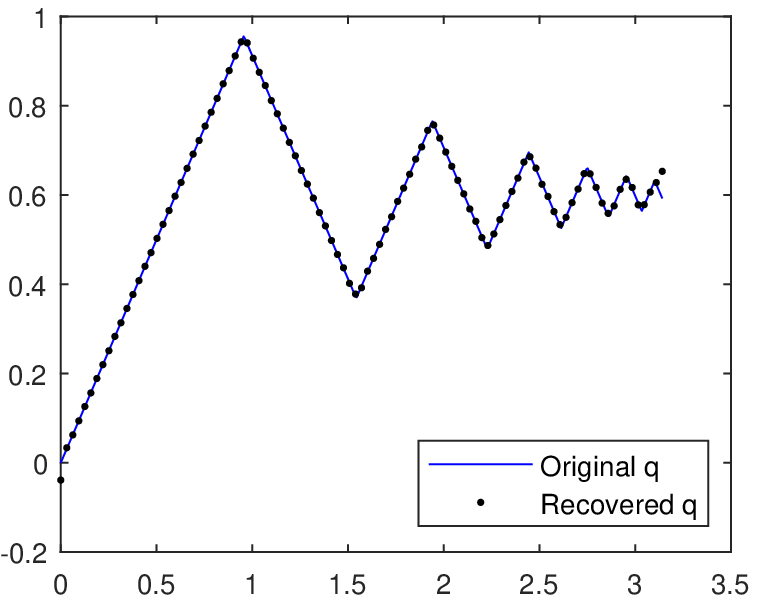}
\end{tabular}
\caption{Exact (blue lines) and recovered (black dots) potentials from Example
\ref{Ex4} together with the number of exact eigenvalue pairs used, average
absolute errors and found values of the parameters $h$ and $H$.}%
\label{Ex4Fig1}%
\end{figure}

Finally, on Figure \ref{Ex4Fig2} on the left we present the original
coefficients $g_{n}(\pi)$ and $h_{n}$ and the ones recovered by solving the
systems \eqref{syst1 overdet} and \eqref{syst2 overdet} and on the right we
present absolute errors of the recovered norming constants.
\begin{figure}[h]
\centering
\includegraphics[bb=0 0 216 173
height=2.4in,
width=3in
]{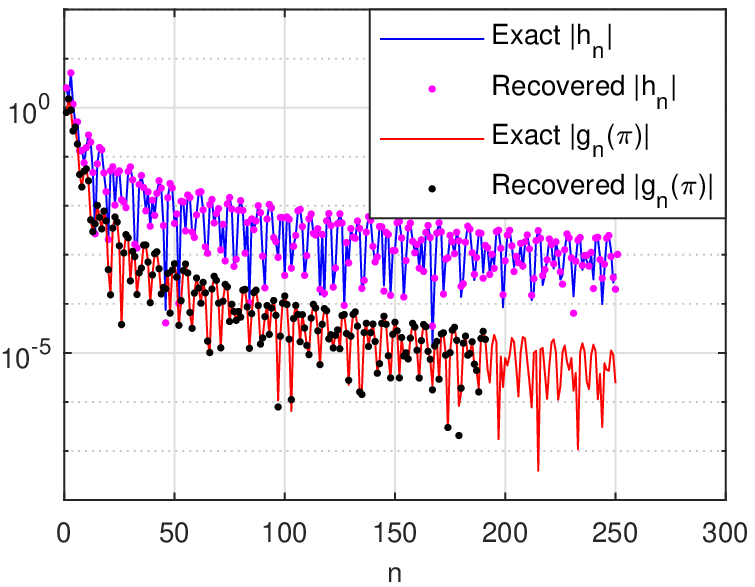} \quad\includegraphics[bb=0 0 216 173
height=2.4in,
width=3in
]{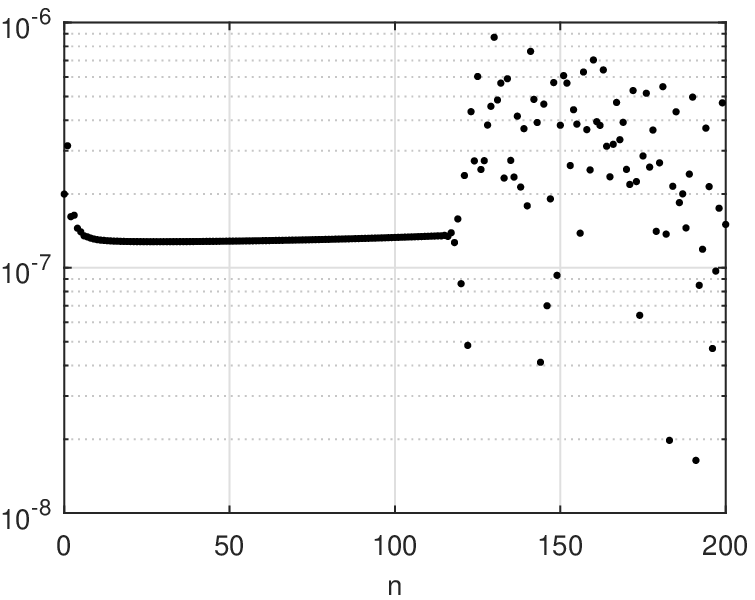}\caption{Details of the work of the proposed algorithm for
the potential $q$ from Example \ref{Ex4}. 201 exact eigenvalue pairs
complemented by 9800 ``asymptotic'' eigenvalue pairs were used. On the left:
the original coefficients $g_{n}(\pi)$ and $h_{n}$, $n\le250$ and the ones
recovered by solving the systems \eqref{syst1 overdet} and
\eqref{syst2 overdet}. On the right: absolute errors of the recovered norming
constants $\alpha_{n}$, $n\le200$.}%
\label{Ex4Fig2}%
\end{figure}
\end{example}

\begin{example}
\label{Ex5} Finally, we consider the spectral problems \eqref{SL equation},
\eqref{bc1} and \eqref{SL equation}, \eqref{bc2} with $h=1$, $H=2$ and
discontinuous potential which we tried to match with those from \cite{Rundell
Sacks}. Namely,
\begin{equation}
q_{3}(x)=%
\begin{cases}
0, & x\in\lbrack0,\frac{\pi}{8}]\cup\lbrack\frac{3\pi}{8},\frac{3\pi}{5}),\\
-\frac{12x}{\pi}+\frac{3}{2}, & x\in(\frac{\pi}{8},\frac{\pi}{4}],\\
\frac{12x}{\pi}-\frac{9}{2}, & x\in(\frac{\pi}{4},\frac{3\pi}{8}),\\
4, & x\in\lbrack\frac{3\pi}{5},\frac{4\pi}{5}),\\
2, & x\in\lbrack\frac{4\pi}{5},\pi].
\end{cases}
\end{equation}
On Figure \ref{Ex5Fig1} we present the recovered potential together with
$L_{1}(0,\pi)$ norm of the absolute error, absolute error of the recovered
parameter $\omega$ and obtained values of the parameters $h$ and $H$ for 30
and 201 eigenvalue pairs. Again, 8 equations were used in the truncated
system. The number of eigenvalues were complemented to 5000 with the
asymptotic eigenvalues.

\begin{figure}[h]
\centering
\begin{tabular}
[c]{cc}%
30 eigenvalue pairs. & 201 eigenvalue pairs\\
$|\omega-\omega_{rec}|= 6.5\cdot10^{-3}$ & $|\omega-\omega_{rec}|=
1.7\cdot10^{-4}$\\
$\|q-q_{rec}\|_{L_{1}(0,\pi)} = 2.2\cdot10^{-1}$ & $\|q-q_{rec}\|_{L_{1}%
(0,\pi)} = 8.7\cdot10^{-2}$\\
$h_{rec} = 1.014$, $H_{rec} = 1.9978$ & $h_{rec} = 0.999966$, $H_{rec} =
2.00060$\\
\includegraphics[bb=0 0 216 173
height=2.4in,
width=3in
]{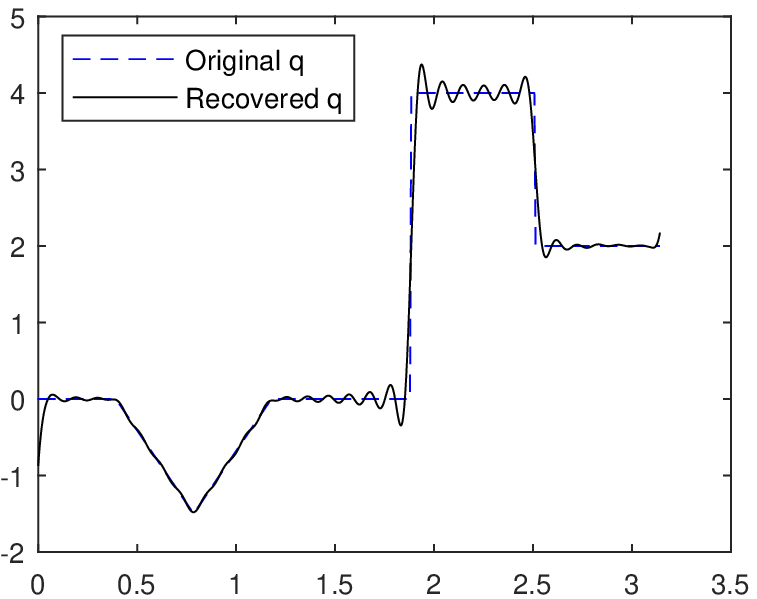} & \includegraphics[bb=0 0 216 173
height=2.4in,
width=3in
]{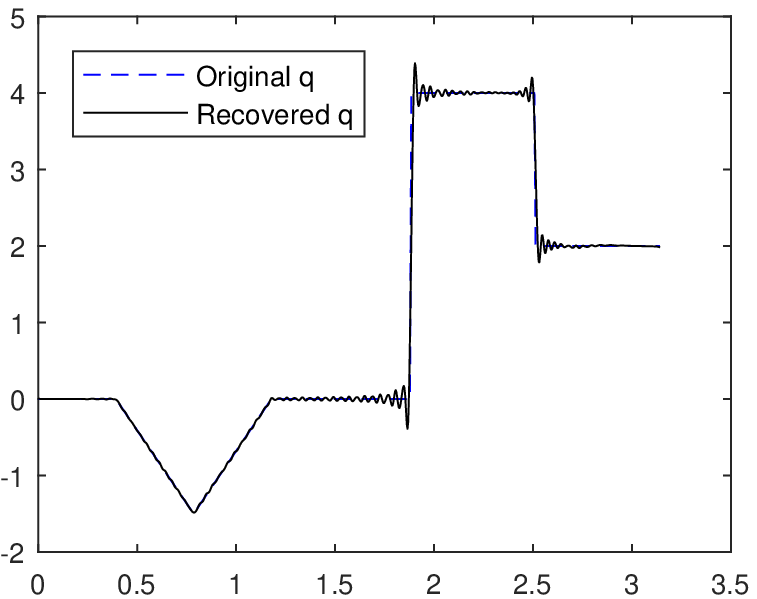}
\end{tabular}
\caption{Exact (blue dashed lines) and recovered (black lines) potential from
Example \ref{Ex5} together with the number of exact eigenvalue pairs used,
average absolute error, error in the recovered parameter $\omega$ and found
values of the parameters $h$ and $H$.}%
\label{Ex5Fig1}%
\end{figure}
\end{example}

\section*{Conclusions}

The method presented here allows one to reduce the inverse spectral problem to
a system of linear algebraic equations, and only the first component of the
solution vector is sufficient for recovering the potential. This, together
with the convergence and stability results, means that a reduced number of
equations (say, five-eight) is enough to obtain an accurate result. Though in
the present work only two classical inverse Sturm-Liouville problems are
considered, we also show how some other inverse Sturm-Liouville problems can
be solved by the same method. The numerical realization of the method is
simple and involves nothing but the build-in functions of a modern numerical
computing environment, such as Matlab.

\end{document}